\documentclass[11pt,twoside,openany]{paper}
\usepackage{
  enumerate,amsmath,graphics,amssymb,graphicx,amscd,xypic,amscd,amsbsy,multirow,
	float,booktabs,verbatim,color,xcolor,amsfonts, latexsym, multirow,times,mathptm, natbib}
\usepackage{subcaption}
\usepackage[top=25mm, bottom=25mm, left=37.5mm, right=37.5mm]{geometry}
\graphicspath{{figs/}{figs2/}{diagrams/}}

\newcommand{\colorone}{red}

\newcommand{\R}{\mathbb{R}}

\newcommand{\E}[1]{\mathbb{E}{ \{ #1 \}}}
\newcommand{\EX}[1]{\mathbb{E}_X{ \{ #1 \}}}

\newcommand{\QED}{\hspace*{\fill}\rule{2.5mm}{2.5mm}}
\newtheorem{theorem}{Theorem}[section]

\newenvironment{proof}{\noindent{\bf Proof\ }}{\QED\\}

\newtheorem{lemma}{Lemma}[section]

\newtheorem{corollary}{Corollary}[section]

\newcommand{\logitinv}{\mbox{logit}^{-1}}
\newcommand{\cov}{\mbox{cov}}

\newcommand{\var}[1]{\mbox{var}\{ #1 \}}

\newcommand{\cor}{\mbox{cor}}
\newcommand{\sd}{\mbox{sd}}
\newcommand*{\pd}[3][]{\ensuremath{\frac{\partial^{#1} #2}{\partial #3}}}
\newcommand{\argmin}{\arg\!\min}

\newcommand{\ybart}{\bar{Y}_t}
\newcommand{\ybarc}{\bar{Y}_c}

\newcommand{\zbart}{\bar{Z}_t}
\newcommand{\zbarc}{\bar{Z}_c}

\newcommand{\hbartzero}{\bar{h}_t(0)}
\newcommand{\hbarczero}{\bar{h}_c(0)}
\newcommand{\hbartone}{\bar{h}_t(1)}
\newcommand{\hbarcone}{\bar{h}_c(1)}

\newcommand{\hbartw}{\bar{h}_t(w)}
\newcommand{\hbarcw}{\bar{h}_c(w)}

\newcommand{\fbart}{\bar{f}_t}
\newcommand{\fbarc}{\bar{f}_c}
\newcommand{\fbaryt}{\bar{f}_{Y, t}}
\newcommand{\fbaryc}{\bar{f}_{Y, c}}
\newcommand{\fbarzt}{\bar{f}_{Z, t}}
\newcommand{\fbarzc}{\bar{f}_{Z, c}}

\newcommand{\taua}{\hat{\tau}_{adj}}

\newcommand{\ystar}{Y^{\star}}
\newcommand{\zstar}{Z^{\star}}
\newcommand{\hstar}{H^{\star}}

\newcommand{\sigmah}{\sigma_{H}}
\newcommand{\sigmay}{\sigma_{Y}}
\newcommand{\sigmahstar}{\sigma_{H^{\star}}}
\newcommand{\sigmaystar}{\sigma_{Y^{\star}}}
\newcommand{\muh}{\mu_{H}}
\newcommand{\muy}{\mu_{Y}}
\newcommand{\muhstar}{\mu_{H^{\star}}}
\newcommand{\muystar}{\mu_{Y^{\star}}}
\newcommand{\mustar}{\mu^{\star}}
\newcommand{\muz}{\mu_{Z}}
\newcommand{\muzstar}{\mu_{Z^{\star}}}
\newcommand{\muhy}{\mu_{H_Y}}
\newcommand{\muhystar}{\mu_{H_Y^{\star}}}
\newcommand{\muhz}{\mu_{H_Z}}
\newcommand{\muhzstar}{\mu_{H_Z^{\star}}}
\newcommand{\rhostar}{\rho^{\star}}
\newcommand{\covyh}{\cov(Y, H)}
\newcommand{\covystarhstar}{\cov(Y^{\star}, H^{\star})}
\newcommand{\sigmahc}{\sigma_{H_c}}
\newcommand{\sigmaht}{\sigma_{H_t}}
\newcommand{\muhc}{\mu_{H_c}}
\newcommand{\muht}{\mu_{H_t}}
\newcommand{\muyc}{\mu_{Y_c}}
\newcommand{\muyt}{\mu_{Y_t}}
\newcommand{\sigmayc}{\sigma_{Y_c}}
\newcommand{\sigmayt}{\sigma_{Y_t}}

\begin{document}

	\begin{center}
		{\Large \bf Unbiased variance reduction in randomized experiments}\\
		\vspace{0.1cm}

	{Reza Hosseini, Amir Najmi}\\
	{Google Inc.}\\
	{\small Address: 1600 Amphitheatre Parkway, Mountain View, CA, 94043, USA}\\
	{email: reza1317@gmail.com, amir@google.com}

	\end{center}

\begin{abstract}
	This paper develops a flexible method for decreasing the
	variance of estimators for
	complex experiment effect metrics (e.g.\;ratio metrics)
	while retaining asymptotic unbiasedness. This method uses
	the auxiliary information
 about the experiment units to decrease the variance.
	The method can incorporate almost any arbitrary predictive model
	(e.g.\;linear regression, regularization, neural networks) to adjust the estimators.
	The adjustment involves some free parameters which can be optimized to achieve
	the smallest variance reduction
	given the predictive model performance.
	Also we approximate the achievable
	reduction in variance in fairly general settings mathematically.
	Finally, we use simulations to show the method works.
\end{abstract}

\section{Introduction}

Online experiments play a pivotal role in decision making for many technology
companies and are widely used across the industry. From a business perspective,
decision makers like to have the fastest results, with the smallest sample of
impacted users. Therefore recently there has been some effort
across the industry to develop experiment effect estimation methods which reduce
the variance e.g.\;\cite{deng-2013}.

Most variance reduction methods come with the price of
introducing bias (which persists even asymptotically)
or making restrictive modeling assumptions. One exception is
discussed in Chapter 7 of \cite{imbens-2015} among other references:
If one defines
the experiment effect to be the simple expected difference
between treatment (denoted by 1) and control (denoted by zero):
\begin{equation}
 	\tau = \E{Y(1)} - \E{Y(0)},
\label{eqn-simple-diff}
\end{equation}
 and $X$ are some auxiliary variables which are
 independent of the assignment procedure
 (assignment of units to control and treatment),
 then the estimated parameter for $\tau_0$ from the linear regression equation
\[Y = \alpha + \tau_{0} W + X \beta + \epsilon,\] where $W=1$ if the unit is in
treatment and zero otherwise, is an asymptotically unbiased estimator of $\tau$
with known asymptotic variance. The proof of this result depends heavily on the
squared loss function properties as well as the simplicity of the experiment effect
considered as apposed to more complex metrics e.g. the ratio metric:
$\tau = \E{Y(1)} / \E{Y(0)}$. However in most cases in the technology industry,
the metrics of interest are ratios rather than simple differences.
(One reason for that being it is easier to think about growth as a percent increase.)
Another limitation of this approach is estimating linear regression coefficients
could be difficult when dealing with a large number of weak predictors which would
have been possible say using regularization methods (e.g.\;see \cite{hastie-2009}).
Our focus in this work is to develop methods which can remove these limitations.

A notable work regarding the simple difference (Equation \ref{eqn-simple-diff}) is
the seminal work of \cite{freedman-2008a, freedman-2008b} which evaluates
regression adjustments using Neyman's nonparametric model (\cite{neyman-1923}). Freedman argues that regression adjustment may introduce bias (although asymptotically unbiased as discussed above); the adjustment may or may not improve precision. As mentioned, here we are mainly concerned with more complex metrics (beyond simple difference) and our approach is not a regression adjustment, but rather uses a predictions model to adjust the estimators. It is worth noting that, using extensions of linear regression to other settings e.g.\; generalized linear models would even sacrifice the asymptotic unbiasedness and that is one of the reasons we do not use regression adjustments for the general settings.

A related paper to this work is \cite{deng-2013} suggesting to use a control variable
(most of the time pre-treatment data on the same variable) which is uncorrelated
with the assignment mechanism to decrease the variance while maintaining
unbiasedness for the simple difference metric $\tau = \E{Y(1)} - \E{Y(0)}$.
We know that an unbiased estimate for this effect is equal to
$\Delta = \bar{Y}_t - \bar{Y}_c$, where $t$ denotes the treatment arm and $c$ denotes
the control arm. Then assume we are given a control (univariate)
variable $X$ for which $\E{X_c} = \E{X_t}$. Then
$\Delta_{cv} = (\bar{Y}_t - \theta \bar{X}_t) - (\bar{Y}_c - \theta \bar{X}_c)$
is an unbiased estimator of $\tau$. Moreover if $X$ and $Y$ have non-zero
correlation, then $\var{Y - \theta X}$ is minimized when
$\theta =  \cov(Y, X)/\var{X}$ in which case the variance shrinks by a
multiplicative factor of $(1 - \rho^2)$, where $\rho$ is the correlation between
$Y$ and $X$. Of course, here we are not dealing with
one variable $Y$ only and rather $Y_t$ and $Y_c$, while we can only pick one $\theta$.
Therefore \cite{deng-2013} suggest to pick a common $\theta$ by using all the
data (which still works well in practice). They also discuss application of this
method to some other more complex metrics for example CTR (click through rate)
using the delta method in Appendix B of the paper. Our method is more general
in several ways. First we replace the variable $X$ with a predictive model
which can take multiple inputs. More importantly, it allows for the experiment effect
form to be the much more general case of
$\tau_g = g(\E{Y(1)}) - g(\E{Y(0)})$ for arbitrary differentiable function $g$ (with domain being the support
of the variables involved),
in particular for $g=\log$ this will capture the ratio metric since
\[\tau_{\log} =  \log(\E{Y(1)}) - \log(\E{Y(0)}) = \log \big( \E{Y(1)}/\E{Y(0)} \big).\]
In this case, we also provide an exact equation for picking the parameter $\theta$.
Moreover we also discuss further generalizations of the method to more complex metrics for example
{\it sum ratio} and {\it ratio of mean ratios} involving more parameters and derive numerical equations for
their optimization.

The paper is organized as follows. Section \ref{sect:define-effect} discusses
the various ways the experiment effect can be defined and the connection between
various definitions.
Section \ref{sect:math-theory} develops the mathematical theory for our method.
Section \ref{sect:sim} performs a simulation study to confirm that the method
retains unbiasedness while decreasing the variance (as much as expected by the theory). Finally Section \ref{section:discussion}
summarizes the results and discussed further potential applications and extensions.

%%%%%%%%%%%%%%%%%
\section{Defining experiment effect}
\label{sect:define-effect}

The experiment effect, denoted here by $\tau$ can be defined in many different
ways depending on the application. Clearly the estimation method and the
confidence interval calculation also depend on this choice and some definitions
may have some advantages over others from a theoretical point of view.
The most common definition is that of the simple difference:
\[\tau = \E{Y(1)} - \E{Y(0)},\]
can be estimated using OLS and (asymptotically unbiased) estimators with smaller variance can be obtained by incorporating
predictor variables $X$ which are independent of the assignment process $W$ into the regression as discussed in Chapter 7 of \cite{imbens-2015}. This simple difference has
been widely considered in the literature from the seminal work of \cite{neyman-1923}
to more recently in \cite{freedman-2008a, freedman-2008b} and \cite{miratrix-2013}.

Here we introduce some general classes of experiment effects and discuss their
interpretation and relationship. We denote the auxiliary information available
to us by $X$ and assume $X$ is independent of the assignment mechanism. $X$ is a
collection of predictive variables which is used in a prediction model to predict the
responses of interest which appear in the experiment effect definition of interest.

One variation we can consider is by taking the expectation of $X$ outside:
\[\tau_X =  \EX{\E{Y(1)|X}}-  \EX{\E{Y(0)|X}},\]
which is equal to $\tau = \E{Y(1)} - \E{Y(0)}$ by the Law of Iterated Expectation.

Here we introduce a generalization of the simple experiment effect defined above.
For any differentiable increasing function $g$, we consider
\[\tau_g = g(\E{Y(1)}) - g(\E{Y(0)}),\]
\[\tau_{g, X} = \mathbb{E}_X \Big \{ g(\E{Y(1)|X}) - g(\E{Y(0)|X}) \Big \}.\]
Note that in general $\tau_{g, X} \neq \tau_g$. In most technology applications,
as far as we know, metrics such  $\tau_{g, X}$ which depend on auxiliary
variables are not considered. However this is a common choice in statistical
literature, observational studies or clinical studies. For example in Chapter 8
of \cite{imbens-2015}, the authors mention that in general we can consider
$\tau$ to be a function of $Y(1), Y(0), X, W,$ i.e.
$\tau = \tau(Y(1), Y(0), X, W)$ rather than $\tau = \tau(Y(1), Y(0))$.
The reason many authors consider this is the simplifications they get by
assuming models. In fact $\tau$ then could become a parameter of a well-studied
model and we give an example of that here for generalized linear models.

Let's
assume that the dynamics of the system and the experiment effect could be
described by a generalized linear model (GLM) with a given link function $g$:
\[g(\E{Y | x, w}) = \alpha + (x - \mu_X)\beta_w + \lambda w,\]
where $\mu_X=\E{X}$ and $\beta_w$ might depend on $w$. Then we have
\[\lambda = \tau_{g, X}.\]
This is because:
\begin{align*}
\tau_{g, X} &= \EX{g(\E{Y(1)|X}) - g(\E{Y(0)|X})}&   \\
&= \EX{\alpha + (X - \mu_X)\beta_1 + \lambda \times 1 - (\alpha + (X - \mu_X)\beta_0 + \lambda \times 0)} & \\
&= \EX{(X - \mu_X)(\beta_1 - \beta_0) + \lambda}  & \\
&= \EX{(X - \mu_X)}(\beta_1 - \beta_0) + \lambda = \lambda & \\
\end{align*}
This implies that if we choose $g$ to be the link function of the GLM, then the
GLM coefficient $\lambda$ is equal to the experiment effect $\tau_{g, X}$ and not
necessarily $\tau_g$. Therefore one cannot extend the regression adjustment approach
(for which the estimated regression coefficient remains an asymptotically unbiased
estimator of $\tau$) to the general case if the purpose is to estimate $\tau_g$.
Note the arbitrary dependence of $\tau_{g, X}$ on $X$ is often an undesirable property for
decision makers.
Unlike regression adjustment, our adjustment method keeps the estimator
asymptotically unbiased as shown in the next section.

In the above we discussed the simple difference metric $\tau$ and its generalization
to $\tau_g = g(\E{Y(1)}) - g(\E{Y(0)})$. However we can also consider ratio
effects e.g.:
\[\tau^r = \E{Y(1)} / \E{Y(0)}.\]
%\[\tau_X^r =  \EX{\E{Y(1)|X}} / \EX{\E{Y(0)|X}} \]
Note that we have
\[g(\tau^r) = \tau_g, \mbox{ for } g=\log,\]
which means we can infer about the ratio metric using an appropriate $g$.
%Note that
%\[\EX{\E{Y(w)|X}} = \E{Y(w)}, \; w=0,1\]
%by the Law of Iterated Expectation and this implies $\tau_{X}^r = \tau^r$.

All the experiment effects above are functions of the mean of treatment and
control variables. While considering the arbitrary $g$ transformation expends
the class of estimators significantly, there are other metrics which are not
covered and we discuss them here.

One popular metric considered in online experiments is ``sum ratio''.
Consider an experiment run on two equally sized arms (with same number of
potential users) with $t$ denoting the treatment and $c$ denoting the control,
then the sum ratio is estimated by
\[
\hat{\tau} =
	\sum_{i \in t} Y_i / \sum_{i \in c}  Y_i =  (n_t/n_c) \bar{Y}_t / \bar{Y}_c,
\]
where $n_t$ and $n_c$ are the number of users with impressions on each arm
and $\bar{Y}_t, \bar{Y}_c$ are the sample averages on each arm.
Usually $Y_i$ is a non-negative variable which can attain zero such as watch
time or money spent. Note that not all the potential users appear in the sum as
some might not have any impressions of the feature. Of course $\hat{\tau}$ is
an estimator and one needs to define a parameter which $\hat{\tau}$ is estimating.
In order to find a reasonable population parameter $\hat{\tau}$ is estimating note
that $\hat{\tau}$ can be written as
\[\hat{\tau} =
	\Big ( (n_t/N) / (n_c/N) \Big ) \bar{Y}_t / \bar{Y}_c =
	(\bar{p}_t / \bar{p}_c) \bar{Y}_t / \bar{Y}_c,\]
where is the (usually unobserved) number of potential users on each arm and
$\bar{p}_t, \bar{p}_c$ are
the probability that a user has impression on treatment and control arms respectively.

Defining $I(W)$ to be an indicator variable to denote if the user on the arm
$W$ has an impression of the feature ($I(W)=1$) or not ($I(W)=0$), we can define
\[\tau = \frac{\E{I(1)}}{\E{I(0)}} \frac{\E{Y(1) | I(1)=1}}{\E{Y(0) | I(0)=1}}\]
and we have
$\hat{\tau} \to \tau$  as desired (by Delta Theorem, see \cite{dasgupta-2008}). Note that with $g$ can also consider
\[\tau_g = g(\E{I(1)} \E{Y(1) | I(1)=1}) - g(\E{I(0)} \E{Y(0) | I(0)=1}),\]
which for $g=\log$ is also equal to
\[\tau_g = g(\E{I(1)}) - g(\E{I(0)}) + g(\E{Y(1) | I(1)=1}) - g(\E{Y(0) | I(0)=1})).\]

Another popular metric considered in online experiment is constructed by
comparing the ratio of two random variables on the treatment arm to the
control arm. More formally consider $Y_i$ and $Z_i$ are two response variables
and consider the ratio of ratios:
\[\hat{\tau} =  \big ( \sum_{i \in t} Y_i / \sum_{i \in t} Z_i \big ) /
\big (\sum_{i \in c} Y_i / \sum_{i \in c} Z_i \big ) =  (\ybart / \zbart) /  (\ybarc / \zbarc) ,\]
which is an estimator for
\[ \tau = \big ( \E{Y(1)} / \E{Z(1)} \big ) / \big ( \E{Y(0)} / \E{Z(0)} \big )d.\]
We could also consider the transformed version for $g=\log$
\[ \tau_g = g(\E{Y(1)}) -  g(\E{Y(0)}) -  g(\E{Z(1)}) -  g(\E{Z(0)}),\]
which has two components, each being a difference in one response. Our method developed here applies also to these more complex cases (sum ratio and ratio of mean ratios).
In the discussion section we give an example of a metric for which our method does
not apply directly.

\section{Mathematical development}
\label{sect:math-theory}

In this section, we develop mathematical theory for reducing variance for
general metrics as discussed in the introduction.
Subsection \ref{sect:math-unbiased} finds necessary and sufficient conditions
for the adjustments to leave the estimators asymptotically unbiased.
Subsection \ref{sect:math-unbiased} uses asymptotic theory to approximate the
achievable variance reduction for one of the adjustment methods which requires
minimal assumption on the predictive model used for adjustment.

%%%%%%%%%%%%%%%%%
\subsection{Retaining asymptotic unbiasedness}
\label{sect:math-unbiased}

In this subsection, we introduce various adjustments and the conditions needed
for each adjustment to avoid introducing bias asymptotically.

In the following, we present two general
ideas to adjust estimators while keeping them unbiased.
Parts (a) and (b) of lemma \ref{lemma-unbiased} use the fact that
that adding a random variable with expectation 0 to an estimator will not introduce bias.
Part (c) notes that if we add a statistic to the estimator with the same expected value,
we can avoid bias by weighting the two terms appropriately.

For simplicity we introduce the following
notations. Consider $Y$ to be a response of interest with two arms,
control denoted by $c$ and treatment denoted by $t$ and let $N_t, N_c$
be the sample size on the respective arms. Then we let
\begin{align*}
\ybart = (1/N_t) \sum_{i \in t} Y_i,\\
\ybarc = (1/N_c) \sum_{i \in c} Y_i,
\end{align*}
which are the average value of the response $Y$ on treatment and control arms
respectively.
Now assume $X_i = (X^1_i, \cdots, X^p_i)$ is a set of $p$ predictive variables
and consider a function $h$ which only depends on $X_i$ and
the experiment arm:
\[h:  \R^p \times [0, 1] \rightarrow \R,\]
where 0 denotes the control and 1 denotes the treatment.
For example $h(X_i, 0)$ is the value of $h$ for given
predictors $X_i$ and for a unit on the control arm.
Then we use the following short hand notations:
\begin{align*}
\hbartw &= (1/N_t) \sum_{i \in t} h(X_i, w),\\
\hbarcw &= (1/N_c) \sum_{i \in c} h(X_i, w),\\
 w &= 0,1
\end{align*}

Also consider a simpler function $f$ which is only a function of the predictors
(and not the experiment arm):
\[f: \R^p \to \R,\]
and consider the short-hand notations:
\begin{align*}
\fbart = (1/N_t) \sum_{i \in t} f(X_i),\\
\fbarc = (1/N_c) \sum_{i \in c} f(X_i).
\end{align*}

Note that while for developing the theory in this section, we do not
require $h, f$ to have any particular properties, in order to get good
adjustments in practice $h, f$ are typically prediction functions which
return the expected value of response given the predictor variables.

\begin{lemma}[Unbiased Adjustment Lemma]
	Suppose $\tau_g = g(\E{Y(1)}) - g(\E{Y(0)})$ be the defined experiment effect.
	Also suppose $\theta \in \R$ and $\gamma_1 + \gamma_2 = 1, \gamma_1,\gamma_2 \in \R.$
	Consider the estimator
	\[\hat{\tau}_g = g(\ybart) - g(\ybarc)\] which is asymptotically unbiased in
	general (and unbiased if $g(x)=x,\;\forall x \in \R$).

	Now assume $h(X_i, W_i)$ and $f(X_i)$ are imputation functions. Then we can consider these
	adjusted estimators:
	\begin{align*}
	& \mbox{(a)} & \; \taua(w) &=
	  \hat{\tau}_g - {\color{\colorone}\theta \big ( g(\hbartw) - g(\hbarcw) \big )},\;w=0,1& \\
  & \mbox{(b)} & \; \taua &=
	  \hat{\tau}_g - {\color{\colorone}\theta \big ( g(\fbart) - g(\fbarc) \big )}& \\
	& \mbox{(c)} & \; \taua &=
	  \gamma_1 \hat{\tau}_g + \gamma_2 \big ( g(\hbarcone) - g(\hbartzero) \big ) \\
	& \mbox{(d)} & \; \taua &=
	  \gamma_1 \big ( g(\ybart) - g(\hbarczero) \big )
	  + \gamma_2 [g(\hbartone) - g(\ybarc) \big )
	\end{align*}
Then
\begin{itemize}
	\item[(a)] $\taua(w)$ is asymptotically unbiased if
	$\E{g(\hbartw)} \to \E{g(\hbarcw)}$ which is also true if
	   $\E{\hbartw} \to \E{\hbarcw}.$
	\item[(b)] $\taua$  is asymptotically unbiased if
	$\E{g(\fbart)} \to \E{g(\fbarc)}$ which is also true if
	   $\E{\fbart} \to \E{\fbarc}.$
	\item[(c)]  $\taua$ is asymptotically unbiased if
	\[\E{g(\hbarcone) - g(\hbartzero)} \to \tau_g \]
	\item[(d)]  $\taua$ is asymptotically unbiased if
	$\E{g(\ybart) - g(\hbarczero)}, \E{g(\hbartone) - g(\ybarc)} \to \tau$.

\end{itemize}
\label{lemma-unbiased}
\end{lemma}

\begin{proof}
	\begin{itemize}
	\item[(a), (b)] First part of each of (a), (b) is obvious. Second part is because of Delta Theorem; see e.g.\;\cite{dasgupta-2008}.
	\item[(c), (d)] Straight-forward from the assumption and linearity of expectation.
\end{itemize}
\end{proof}

\noindent{\bf Remark.} Note that (c), (d) are in general much stronger conditions
than (a) and (b).

\noindent{\bf Remark on generality of the method in Parts (a) , (b)} Parts (a) and (b) open
the door for very flexible paradigm for variance reduction with a very wide class
of predictive models since
\[\E{\hbartw} \to \E{\hbarcw},\]
holds for almost any predictive model e.g.\;generalized linear models,
regularization methods etc.\\

\noindent{\bf Remark on related alternatives to (a) and (b)}
We can consider related alternatives to (a), (b). E.g.\;consider
\[
	\taua =
	\hat{\tau}_g
	- {\color{\colorone} \sum_{w \in \{0, 1\}} \theta_w \big ( g(\hbartw) -
g(\hbarcw) \big )}.
\]
In this case we are adjusting the estimator by removing two terms which have
asymptotic expectation of zero: one is the difference of the $g$-transformed
average model prediction on the two arms assuming both were control; and one is
the $g$-transformed averages on the two arms assuming both were treatment.
Another alternative can be achieved by considering
\[
  \hat{\tau}_{adj} =
    \hat{\tau}_g - {\color{\colorone} \theta [g(\gamma_1 \hbartzero + \gamma_2 \hbartone) - g(\gamma_1 \hbarczero + \gamma_2 \hbarcone)]}\]
for $\gamma_1 + \gamma_2 = 1, \gamma_1,\gamma_2 \geq 0$. Now if we define
$f(X_i) = \gamma_1 h(X_i, 0) + \gamma_2 h(X_i, 1)$ which is a weighted average of
prediction of $h$ on control and treatment, we get
\[
	\hat{\tau}_{adj} =
	\hat{\tau}_g
	- {\color{\colorone}\theta (g(\fbart) - g(\fbarc))},
\]
which is a special case of Part (b).
We could interpret $f$ as a prediction function which tries to predict the
value of the response well for both arms. This can motivate us to fit a
prediction function to all data (including experiment and control) by ignoring
the label as suggested by \cite{deng-2013}.\\

\noindent{\bf Remark on training the  data.} Note that even though in Part
(a) for $w=0$, the model is predicting the values assuming the data is on the control arm,
we are not required to use the control data only and we can use all the available data to fit a model and predict assuming all data is on the control arm.

The following corollary explicitly states the case for the ratio metrics, by
considering and appropriate $g$.
\begin{corollary}
	Suppose for non-negative valued response $Y$ with positive expectation, we are interested in
	$\tau^r = \E{Y(1)} / \E{Y(0)}$. Then
	\[\hat{\tau} = (\ybart /  \ybarc)\]  is an asymptotically unbiased estimate
	of $\tau^r$. Also
	\begin{itemize}
		\item[(a)] For either of $w=0,1$, we have
		\[\taua(w) = \hat{\tau} \times
		{\color{\colorone} \big ( \hbartw/\hbarcw \big )^{\theta}}\]
		is asymptotically
		unbiased if
		\[\E{\hbartw/\hbarcw} \to 1.\]
		\item[(b)]  $\taua = \hat{\tau} \times
		{\color{\colorone} \big ( \fbart/\fbarc \big )^{\theta}}$ is asymptotically
		unbiased if
		\[\E{\fbart/\fbarc} \to 1.\]
		\item[(c)] $\taua = \hat{\tau}^{\gamma_1}  \times  (\hbarcone / \hbartzero)^{\gamma_2}$ where $\gamma_1+\gamma_2=1$ if
		$\hbarcone / \hbartzero \to \tau$.
		\item[(d)] $\taua = (\ybart / \hbarcone)^{\gamma_1}(\hbartzero / \ybarc)^{\gamma_2}$ where $\gamma_1+\gamma_2=1$ if
		$(\ybart / \hbarcone),  (\hbartzero / \ybarc) \to \tau$.
	\end{itemize}
\label{corol-ratio}
\end{corollary}

\begin{proof}
	 In order to infer about $\tau^r$ we can infer about
	\[\log(\tau^r) = \log(\E{Y(1)}) - \log(\E{Y(0)})\]
	which is the same as $\tau_g = g(\E{Y(1)}) - g(\E{Y(0)})$ when $g=\log$.
\end{proof}

%%%%%%%%%%%%%%%%%
\subsection{Reduction optimization and approximating achievable reduction}
\label{sect:math-var-reduction}

This subsection uses asymptotic theory to approximate the reduction in the
estimator variance with the methods suggested above. We only work out the case
for parts (a), (b) of Lemma \ref{lemma-unbiased}:
\[
	\taua(w) =
		\hat{\tau}_g - {\color{\colorone}\theta \big ( g(\hbartw) - g(\hbarcw) \big )},\;
		w=1,2
\]
and
\[\taua = \hat{\tau}_g - {\color{\colorone}\theta (g(\fbart) - g(\fbarc) \big )}.\]

In this case we find an optimal $\theta$ (as a function of $g$ and the
predictive power of $h$) to achieve maximum reduced variance.

We do not pursue parts (c), (d) type estimators as they require much stronger
assumptions on the model to retain unbiasedness.
The main result of this subsection is given in Theorem \ref{theo-var-g2} which
relies on the Delta Theorem in the multivariate case (see \cite{dasgupta-2008}).
Much more general results are proved in Section
\ref{subsection-adj-complex-metrics}.
However the simpler cases in this section suffice for most applications. Also we are able
to provide more detailed solutions for these simpler cases.

First in Theorem \ref{theo-var-g}, we find the minimizer ($\theta$) for the variance of
statistics of the form $T = g(Y) - \theta g(H)$. This is already a generalization of
the main idea in \cite{deng-2013}. However our adjusted estimators discussed in Lemma
\ref{lemma-unbiased} are more complex and involves two differences added to each other i.e. of the form
\[T = g(Y) - \theta g(H) - \big ( g(\ystar) - \theta g(\hstar) \big ),\]
and Theorem \ref{theo-var-g2} optimizes for $\theta$ for this case.

\begin{theorem}
	Consider the random variable $T = g(Y) - \theta g(H)$ where $H, Y$ are random
	variables with finite moments and $g$ is a real-valued differentiable function
	and $g'(\mu_H) \neq 0$.  Also assume $\rho = \cor(Y, H)$.
	Then
	\begin{itemize}
		\item[(a)]
		 \[\argmin_{\theta} \var{T} \approx \big ( g'(\mu_Y) / g'(\mu_H) \big )  \cov(Y, H) / \sigma_H^2 \]
		\item[(b)] The minimum is then approximated by
		\[(g'(\mu_Y) \sigma_Y)^2 (1 - \rho^2).\]
		Moreover
		\[\var{g(Y) - \theta g(H)} / \var{g(Y)} \approx (1 - \rho^2)\]
		which does not depend on $g$.
   \end{itemize}
\label{theo-var-g}
\end{theorem}

\begin{proof}
See Appendix.
\end{proof}

\begin{corollary}
	For $g=\log$, the argmin is equal to $(\mu_H/\mu_Y)  \cov(Y, H) / \sigma_H^2.$
\end{corollary}

\begin{theorem}
	Consider the random variables

	\begin{align*}
		T_0 =\; & g(Y) - g(\ystar), & \\
		T_1 =\; & g(Y) - \theta g(H), &\\
		T_2 =\; & g(\ystar) - \theta g(\hstar), &\\
		T =\; & T_1 - T_2 = g(Y) - \theta g(H) - \big ( g(\ystar) - \theta g(\hstar) \big ) & \\
	\end{align*}
	where $Y,H, \ystar,\hstar$ are random variables with
	finite moments and $g$ is a real-valued differentiable function.
	Let $\mu=(\muy, \muh, \muystar, \muhstar)$ be the mean vector and
	$(\sigmay, \sigmah, \sigmaystar, \sigmahstar)$ be the standard deviation vector.
	Also assume all the pairwise correlations are zero except for
	$\covyh \mbox{ and } \covystarhstar$ and
  let $\rho = \cor(Y, H)$ and $\rho^{\star} = \cor(\ystar, \hstar)$. Then
	\begin{itemize}
		\item[(a)]
		 \[\theta := \argmin_{\theta} \var{T} \approx
		 \frac{g'(\muh)g'(\muy) \cov(Y, H) + g'(\muhstar)g'(\muystar) \cov(\ystar, \hstar)}
		 {g'(\muh)^2 \sigmah^2 + g'(\muhstar)^2 \sigmahstar^2}.\]

		 For $g$ = identity, we have
		 \[\theta = \frac{\cov(Y, H) + \cov(\ystar, \hstar)}{\sigmah^2 + \sigmahstar^2}.\]
		 For $g=\log$, we have
		 \[\theta = \frac{\cov(Y, H)/(\muh\muy) + \cov(\ystar, \hstar)/(\muhstar\muystar)}{\sigmah^2/\muh^2 + \sigmahstar^2/\muhstar^2}.\]

		\item[(b)] The minimum is then approximated by
		\[g'(\muy)^2 \sigmay^2 + g'(\muystar)^2 \sigmaystar^2 - \delta,\]
		 where
		 \[\delta = \frac{\big ( g'(\muh)g'(\muy)\covyh +
		 g'(\muhstar)g'(\muystar)\covystarhstar \big )^2}{g'(\muh)^2\sigmah^2 + g'(\muhstar)^2 \sigmahstar^2}\]
  Therefore $\var{T_0} - \var{T} \to \delta$.

  For $g$ = identity, we have
  \[\delta = \frac{\big ( \cov(Y, H) + \cov(\ystar, \hstar) \big )^2}{\sigmah^2 + \sigmahstar^2}.\]
  For $g=\log$, we have
		 \[\delta = \frac{\big ( \cov(Y, H)/(\muh\muy) +
		 \cov(\ystar, \hstar)/(\muhstar\muystar)\big )^2}{\sigmah^2/\muh^2 + \sigmahstar^2/\muhstar^2}.\]

  \item[(c)]  $\min\{\theta_1, \theta_2\} \leq \theta \leq \max\{\theta_1, \theta_2\},$
  where $\theta_i$ is the argmin for minimizing the variance for $T_i,\; i=1,2$.

  \item[(d)] If we further assume $g'(\muh) \sigmah = g'(\muhstar) \sigmahstar$
  and $\rho g'(\muy)\sigmay = \rho g'(\muystar)\sigmaystar$, we have
  \[\var{T} =
  	(1-\rho^2) \var{Y} +
  	(1-(\rhostar)^2) \var{\ystar} \leq (1 - \min\{\rho^2, {\rhostar}^2\}) \var{T_0}\]

   \end{itemize}
\label{theo-var-g2}
\end{theorem}

\begin{proof}
See Appendix.
\end{proof}

Now let's apply this theorem to our problem.

\begin{corollary}[Variance Reduction]
	Suppose $\tau_g = g(\E{Y(1)}) - g(\E{Y(0)})$ be the defined experiment effect.
	Also consider the consistent (raw) estimator
		\[\hat{\tau}_g = g(\ybart )- g(\ybarc).\]
	Assume $f(X_i)$ is an interpolation function which does not depend on the
	experiment arm, e.g. $f(X_i) = h(X_i, 0)$ and consider the adjusted estimator
	\[\hat{\tau}_{adj} = \hat{\tau}_g - {\color{\colorone}\theta \big ( g(\fbart) - g(\fbarc) \big ) }.\]
	We denote $f(X_i)$ by $H_i$ for simplicity. Also we denote $(Y_i, f(X_i))$ when
	$i \in t$ (chosen at random) by $(Y_t, H_t)$ and $(Y_c, H_c)$ for the control arm.
	Let
	\[\rho_{c} = \cor(Y_c, H_c)\;\; \rho_t = \cor(Y_t, H_t)\]
	Then
	\begin{itemize}
	\item[(a)] the optimal $\theta$ which minimizes $\taua$ is between $\theta_1$
	and $\theta_2$ which minimize the variance of
	$g(\ybarc) - \theta g(\fbarc)$ and $g(\ybart) - \theta g(\fbart)$ respectively:
	\[\theta_1 = \cov(Y_c, H_c) / \var{H_c}, \mbox{ and } \theta_2 = \cov(Y_t, H_t) / \var{H_t}\]

	\item[(b)] Moreover if we assume
	\[g'(\muhc) \sigmahc = g'(\muht) \sigmaht \mbox{ and } \rho_c g'(\muyc)\sigmayc = \rho_t g'(\muyt)\sigmayt\]

	\[\var{ \taua } / \var{ \hat{\tau} } \leq (1 - \min \{\rho_c^2, \rho_t^2\}).\]
	\end{itemize}

	\label{corol-var-reduction}
	\end{corollary}

\begin{proof}
See Appendix.
\end{proof}

{\noindent \bf Remark.} The assumptions is Part (b) are only used to
approximate the decrease in the variance and are not needed for the method to
work. Even if these assumptions do not hold, we can expect a decrease in
variance. Moreover, in most practical cases, these assumption approximately
hold. For example in most useful interpolation models we expect even stronger
assumptions to hold e.g. we expect $\muht = \muhc$ meaning that the model
predictions on the experiment and control arm are the same on average, which
must be true for most models considering the experiment units are chosen at
random and independent from their $x_i$  which is used by the interpolation
function $f$.

The reduction in the confidence interval length, can then be approximated by
\begin{equation}
	\mbox{(raw CI length - adj CI length)}/ \mbox{raw CI length} \approx  1 - \sqrt{1 - \rho^2}
	\label{eqn-var-reduction}
\end{equation}

Figure \ref{fig:ci_reduct}, depicts this relationship. For example with a
correlation of 0.6 we can expect 20\% reduction and
with a correlation of 0.8 we can expect 40\% reduction and with 0.9 correlation
we get close to 50\%. To get an idea for cost savings, lets compare this to the
reduction achieved by increasing the sample size from $n$ to $kn$ for $k>1$. In
this case the reduction is equal to
\[
	\mbox{(raw CI length - adj CI length)}/ \mbox{raw CI length} =
	\frac{\sigma/\sqrt{n} - \sigma/\sqrt{kn}}{\sigma/\sqrt{n}} =
	1 - 1/\sqrt{k}
\]
This means for a 20\% reduction we need to multiply sample size by 1.5
(as compared with 0.6 correlation with adjustment) and to get 40\% reduction
need to almost triple the sample size.

\begin{figure}[H]
	\centering
	\includegraphics[width=0.3\linewidth]{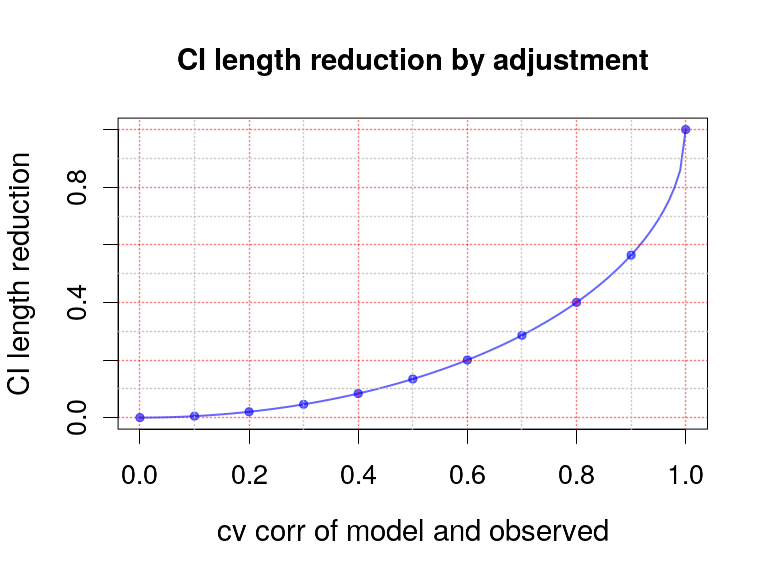}
	\includegraphics[width=0.3\linewidth]{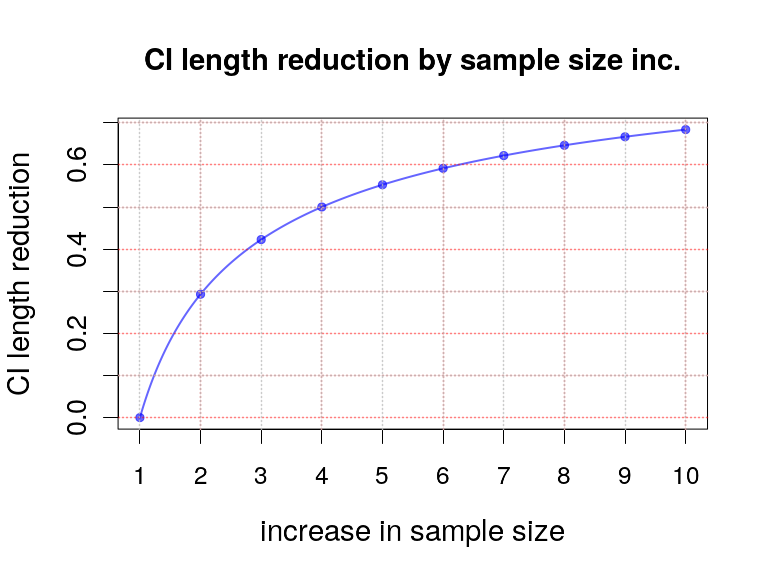}
	\includegraphics[width=0.3\linewidth]{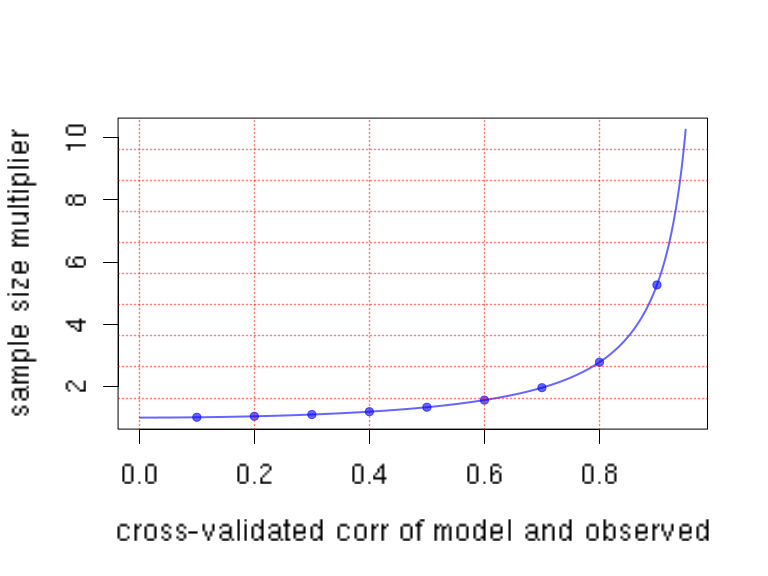}\
	\caption{Left: The reduction in CI (confidence intervals) length as a function
	of cross-validated correlation between predictions and observed. Middle:
	The reduction in CI length as a function of increase in sample size.
	Right: The practical sample size gain compared to the cross-validated correlation.}
	\label{fig:ci_reduct}
\end{figure}

\subsection{Adjusting other complex metrics}
\label{subsection-adj-complex-metrics}

In this subsection we discuss using these methods for other complex metrics
discussed at the end of Section \ref{sect:define-effect}.

First consider the sum ratio:
\[
	\hat{\tau} =
	\sum_{i \in t} Y_i / \sum_{i \in c}  Y_i =  (n_t/n_c) \ybart / \ybarc,
\]
and note that since we are assuming not all potential users appear in the sum,
$n_t$ and $n_c$ can be thought of random variables themselves
(if they are fixed variables in the experiment by design then we can simply
revert back to the mean ratio case).
Considering the second term in the above is simply the estimator for the mean
ratio, we can replace the second term by its adjusted version by multiplying the
estimator by  $[\fbart/\fbarc]^{\theta}$ which leaves the estimator
asymptotically unbiased. Therefore in summary we can use the exact same
adjustment as the mean ratio for the sum case. In the appendix we have performed
simulations for this case.

Next consider the ratio of ratios metrics. The estimator in that case in the log
scale is given by
\[\hat{\tau}_g =   g(\ybart) -  g(\zbart)   -  \big (g(\ybarc)  -  g(\zbarc) \big ) \]
\[=  g(\ybart) -  g(\ybarc)   -  \big ( g(\zbart) -  g(\zbarc) \big ).\]
For general $g$ the above expression could be considered as a {\it $g$-transformed difference
of difference} which has many applications.
For example when we compare the experiment and
treatment arms difference with a baseline.

Now we can adjust each of the two components on the right side individually
without introducing bias as
\[
	\hat{\tau}_{adj} =
	g(\ybart) - g(\ybarc)
	- \theta_Y (g(\fbaryt) - g(\fbaryc))
	- \Big ( g(\zbart) -  g(\zbarc)
	- \theta_Z (g(\fbarzt) - g(\fbarzc)) \Big ).
\]

This implies an adjusted estimator for $\tau$ is
\[
	(\ybart / \ybarc)(\zbarc / \zbart)
	[\fbaryc/\fbaryt)]^{\theta_Y}
	[\fbarzt/\fbarzc)]^{\theta_Z}.
\]

Note that we do not need to require $\theta_Y = \theta_Z$ of course and could
pick the individual $\theta$ which minimizes each term separately. However
that would not necessarily give us the global argmin. Indeed it is possible to find
the optimal values for this problem and much more general class of metrics
(using linear algebra) and we publish those results in future manuscripts.

\section{Simulations}
\label{sect:sim}

This section, performs a simulation study to test the variance reduction method
in Part (a) of Lemma \ref{lemma-unbiased} with $g=\log$ which is equivalent to the
Mean Ratio metric in which case the raw estimator is equal to
\[\hat{\tau} = (\sum_{i \in t} Y_i(1)/N_t) /  (\sum_{i \in c} Y_i(0)/N_c),\]
and the adjusted is equal to
\[\hat{\tau}_{adj} = \hat{\tau} \times
		{\color{\colorone} [(\sum_{i \in t} h(X_i, 0) / N_t)/(\sum_{i \in c} h(X_i, 0) / N_c)]^{\theta}}.\]
We examine the existence of bias in the new method and the amount of variance
reduction which is compared to the theoretical reduction approximated in Equation
Corollary \ref{corol-var-reduction}. More simulation studies are performed for
various other scenarios and the results are included in the appendix.
We generate a large population of users (1000,000) to play the role of the whole
(super) population and use smaller sample sizes to test the methods.

In this simulation, we assume a set of users are exposed to different versions of
the same feature. Each user may have a different  number of {\it impressions}
(denoted by $Imp$ in the following) to the feature. In each impression the user
has a chance to spend certain {\it amount} denoted by $A$) of time/money on the
feature. Also the user have a chance to either have an explicit interaction
(denoted by $Interact$) with the feature or not each time.

We first simulate a 100,000 population of users with random user attributes
(country and gender) and then for those users, we simulate impression counts
using these model assumptions:
\[Imp(u) \sim ZTP(\mu_1),\]
\[ \mu_1 =  \exp(X(u)\beta_1 + \theta_1 W(u) + \epsilon_1(u)),\]
where $W(u)=1$ if $u$ is in the treatment arm and zero otherwise; ZTP stands for
Zero-Truncated Poisson. The reason we use ZTP instead of Poisson is to avoid
having user imbalance in the predictors (country and gender here) the two arms
due to experiment which would then result in bias.

We use country and gender as our predictors here.
Figure \ref{fig:check-for-population-imbalance-v15} depicts the number of users for
each slice of the predictors (country/gender) and for each arm. At the population
level we observe approximate balance in how the users in each arm are distributed
across country and gender which means the appearance of users in the data from
various slices of the prediction variables is not impacted by the experiment as
desired.

Next conditional on the impression, we simulate interactions and amounts.
For amount
\[A(u) | Imp(u) \sim exp(\mu_2)\]
\[ \mu_2 =  \exp(X(u)\beta_2 + \theta_2 W(u) + \epsilon_2(u))  \]
and for the interaction:
\[Interact(u) | Imp(u) \sim Bernoulli(\mu_3)\]
\[ \mu_3 =  \logitinv(X(u) \beta_3 + \theta_3 W(u) + \epsilon_3(u)),\]
where $\logitinv(x) = \exp(x) / (1 + \exp(x))$. Also $\epsilon_1(u), \epsilon_2(u), \epsilon_3(u)$ model the user random effect
and assumed to be multivariate normal.

\begin{figure}[H]
	\centering
	\includegraphics[width=0.5\linewidth]{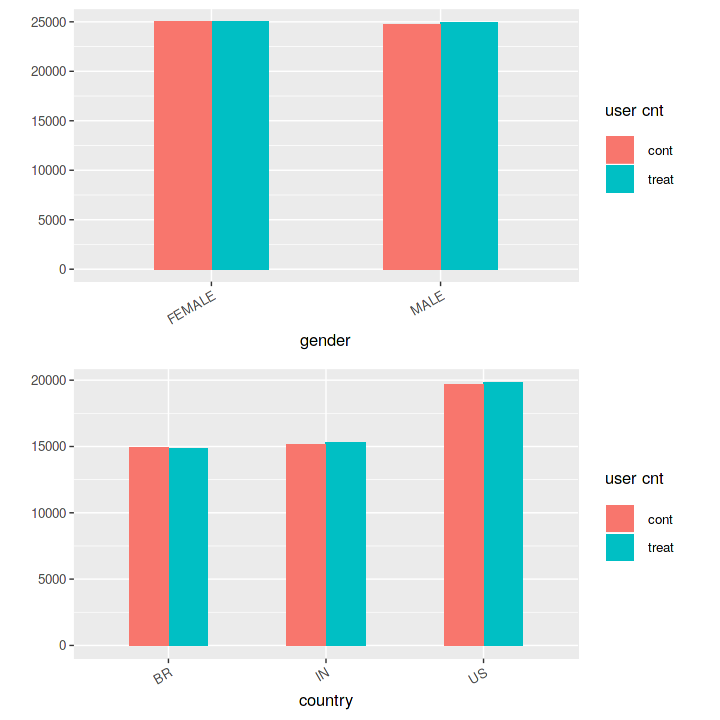}
	\caption{Check for population imbalance across predictors}
	\label{fig:check-for-population-imbalance-v15}
\end{figure}

We consider the three response variables: impression, amount and interaction
and the predictors: country and gender.
To test our method, we consider the Mean Ratio Metric for each response variable:
\[
	\tau =
	\E{Y(1)} / \E{Y(0)},
\]
defined over the (super) population of all units eligible and $t,c$ denote
treatment and control respectively.
Then assuming we have data from a random experiment, we can estimate $\tau$
by the following unbiased estimator which we call the raw method:
\[\mbox{raw estimate} = (\sum_{i \in t} y_i / n_t) / (\sum_{i \in c} y_i / n_c),\]
where $n_t, n_c$ denote the sample size on the treatment and the control arm
respectively.

Then to calculate adjusted estimators we use linear regression to fit a
predictive model for each response. We fit two regression models for each response:
\begin{itemize}
	\item[(i)] using control data only: the data is fitted only to the control arm data
	\item[(ii)] using all data (including treatment data): the data is fitted to
	both arms, using the experiment label as a predictor -- however when predicting,
	we assume all units are on control arm.
	\item[(iii)] using all data but dropping the experiment label as a predictor
\end{itemize}
and we use the multiplicative adjustment from Corollary \ref{corol-ratio}, Part (a):
\[\color{\colorone} a = (
		\sum_{i \in t} h(x_i; 0) / n_t) /
		(\sum_{i \in c} h)(x_i; 0) / n_c),
\]

where $h$ denotes the model prediction function and $x_i$ denotes the value of
the predictors for user $i$. Also note that the 0 in the
second component of $ h(x_i; 0)$ prescribe to the model to assume the user has
been on the control arm. This is automatically the case for
(i), but for two if we use the experiment arm as a categorical variable in the
model, when performing the interpolation we need to make sure to predict the
response for all users assuming they were on the control arm. The results from
these two models are given Tables
	\ref{pred_accuracy_infoDf_contDataOnly_v15},
	\ref{pred_accuracy_infoDf_withTreatData_v15} and
	\ref{pred_accuracy_infoDf_allDataNoExptId_v15}
respectively. The
cross-validated correlation for each response is given in the column $cv\_cor$,{}
the optimal $\theta$
 and the theoretical reduction in the variance are also given. There is little
 difference between the model performance in this example and we expect both
 adjustment methods produce similar reduction.

%pred_accuracy_infoDf_withTreatData_v15
%pred_accuracy_infodf_contdataOnly_v15
% latex table generated in R 3.4.1 by xtable 1.8-2 package
% Tue Jan 15 07:39:47 2019
\begin{table}[ht]
\centering
\begin{tabular}{|r|l|r|r|r|}
  \hline
 & model\_formula & cv\_cor & theta & sd\_ratio \\
  \hline
1 & imp\_count: gender+country & 0.67 & 0.97 & 0.74 \\
  2 & obs\_interact: gender+country & 0.75 & 1.00 & 0.66 \\
  3 & obs\_amount: gender+country & 0.72 & 1.00 & 0.69 \\
   \hline
\end{tabular}
\caption{Prediction accuracy for various responses when using control data only.} 
\label{pred_accuracy_infoDf_contDataOnly_v15}
\end{table}

% latex table generated in R 3.4.1 by xtable 1.8-2 package
% Tue Jan 15 07:39:47 2019
\begin{table}[ht]
\centering
\begin{tabular}{|r|l|r|r|r|}
  \hline
 & model\_formula & cv\_cor & theta & sd\_ratio \\
  \hline
1 & imp\_count: gender+country+expt\_id & 0.69 & 0.99 & 0.72 \\
  2 & obs\_interact: gender+country+expt\_id & 0.74 & 1.00 & 0.67 \\
  3 & obs\_amount: gender+country+expt\_id & 0.66 & 1.00 & 0.75 \\
   \hline
\end{tabular}
\caption{Prediction accuracy for various responses when using all the data in building the model and the experiment label. However we assume all data is on control arm when predicting.} 
\label{pred_accuracy_infoDf_withTreatData_v15}
\end{table}

% latex table generated in R 3.4.1 by xtable 1.8-2 package
% Tue Jan 15 07:39:47 2019
\begin{table}[ht]
\centering
\begin{tabular}{|r|l|r|r|r|}
  \hline
 & model\_formula & cv\_cor & theta & sd\_ratio \\
  \hline
1 & imp\_count: gender+country & 0.65 & 1.00 & 0.76 \\
  2 & obs\_interact: gender+country & 0.69 & 1.00 & 0.72 \\ 
  3 & obs\_amount: gender+country & 0.62 & 1.00 & 0.78 \\
   \hline
\end{tabular}
\caption{Prediction accuracy for various responses when using all the data in building the model but not using the experiment label.}
\label{pred_accuracy_infoDf_allDataNoExptId_v15}
\end{table}

Figure \ref{fig:check_for_imbalance_v15} checks if the adjustment
(in the log scale) is on average unbiased.
To that end, we sample 500 users from all the data randomly multiple times and
calculate the adjustment in each case.
We expect to see a distribution which is symmetric around zero in the absence
of bias. We observe that in all cases the
adjustment is unbiased.

\begin{figure}[H]
	\centering
	\includegraphics[width=0.8\linewidth]{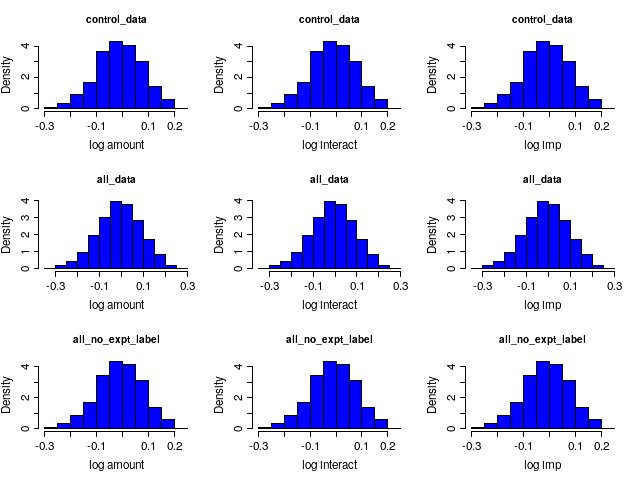}
	\caption{Check for imbalance in the adjustment when $g=\log$.
	The top panels are using the control data only. The middle plots use all
	the data and use the experiment label in training.
	The bottom plots are using all data but do not use the experiment label.}
	\label{fig:check_for_imbalance_v15}
\end{figure}

Next we vary the sample size of users from 500 to 20,000 and for each sample size,
we sample from the population 1000 times to calculate the estimators 1000 times
for both raw and adjusted cases.
Figure \ref{fig:mean_ratio_v15} depicts the mean of the estimator for the three
responses on the left panels and the variance on the right panels.
We only show adjustment methods based on (i) and (ii) as (iii) was very similar to (ii).
It is clear that both adjustment methods leave the estimator unbiased while reducing the
variance significantly regardless of the sample size.
Figure \ref{fig:mean_ratio_sd_comparison_v15} depicts the ratio of the standard
deviation adjusted estimators to the  raw estimator which is approximated to be
\[\mbox{ratio} =
	\sd(\mbox{adj estimator}) /
	\sd(\mbox{raw estimator}) \approx \sqrt{1 - \rho^2},
\]
using Theorem \ref{theo-var-g} and its corollaries. For $\rho \approx 0.70$ this
ratio is approximately 0.71 which is what we observe in the figure.
Note that this ratio does not depend on the sample size according to the theory
and the figure confirms that for this case. Other simulation results for varying
correlation values (presented in the appendix) were also consistent with the theory.

\begin{figure}[H]
	\centering
	\includegraphics[width=0.4\linewidth]{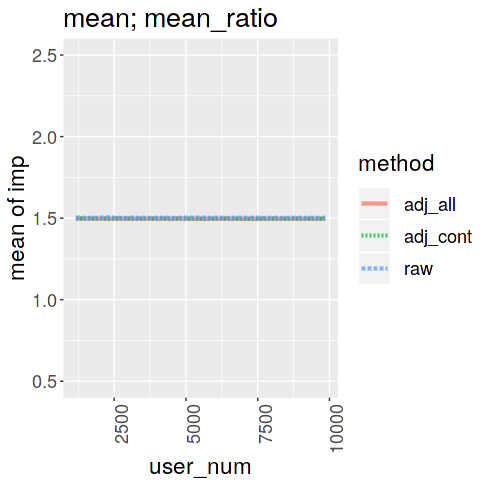}
	\includegraphics[width=0.4\linewidth]{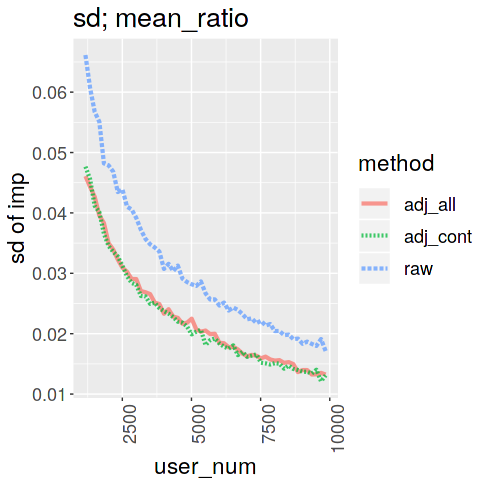}

	\includegraphics[width=0.48\linewidth]{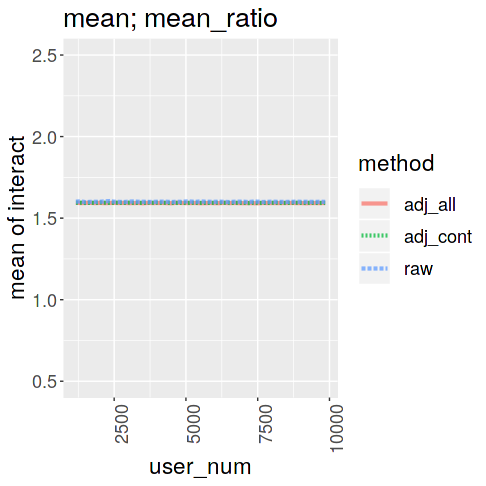}
	\includegraphics[width=0.48\linewidth]{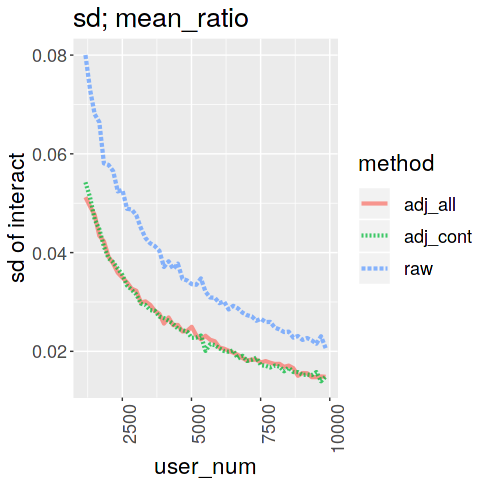}

	\includegraphics[width=0.4\linewidth]{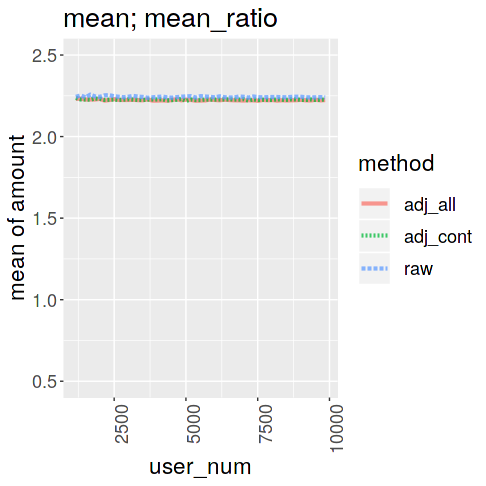}
	\includegraphics[width=0.4\linewidth]{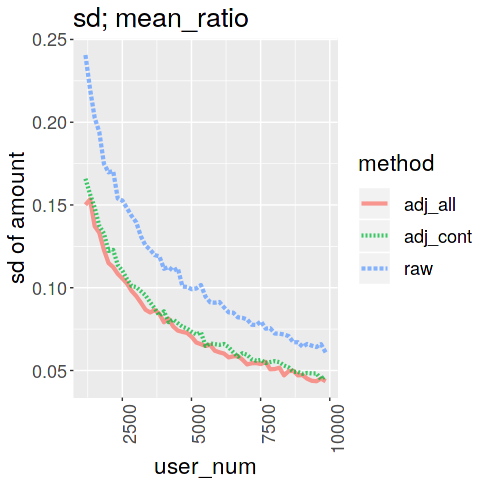}
	\caption{Mean Ratio Metric estimator's mean and variance across sample size.}
	\label{fig:mean_ratio_v15}
\end{figure}

\begin{figure}[H]
	\centering
	\includegraphics[width=0.8\linewidth]{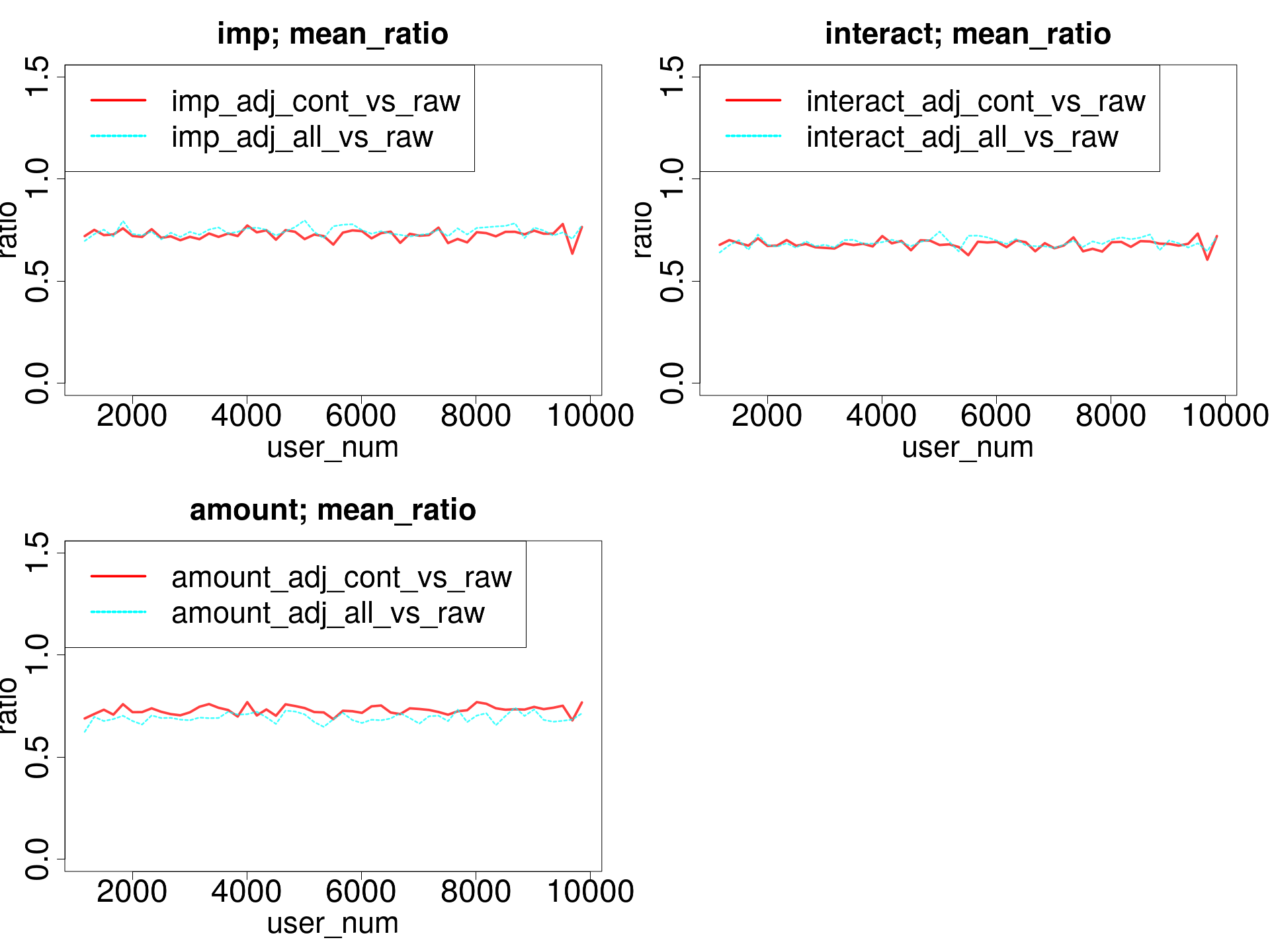}
	\caption{The ratio of standard deviations of adjusted and raw estimators for
	Mean Ratio.}
	\label{fig:mean_ratio_sd_comparison_v15}
\end{figure}

Finally, Figure \ref{fig:mean_ratio_ci_convg_comparison_v15} varies the sample
size from 500 to 10000. For each sample size, we sample from the population once,
and for that one sample we calculate a CI using the bucketed JackKnife method
with 50 buckets. The bucketed jack-knife method is similar to regular jack knife
method but takes out one bucket for each calculation, rather than one unit.
We observe that the adjusted confidence intervals are within
the raw confidence intervals most of the time, while all converging toward a
common value.

\begin{figure}[H]
	\centering
	\includegraphics[width=0.999\linewidth]{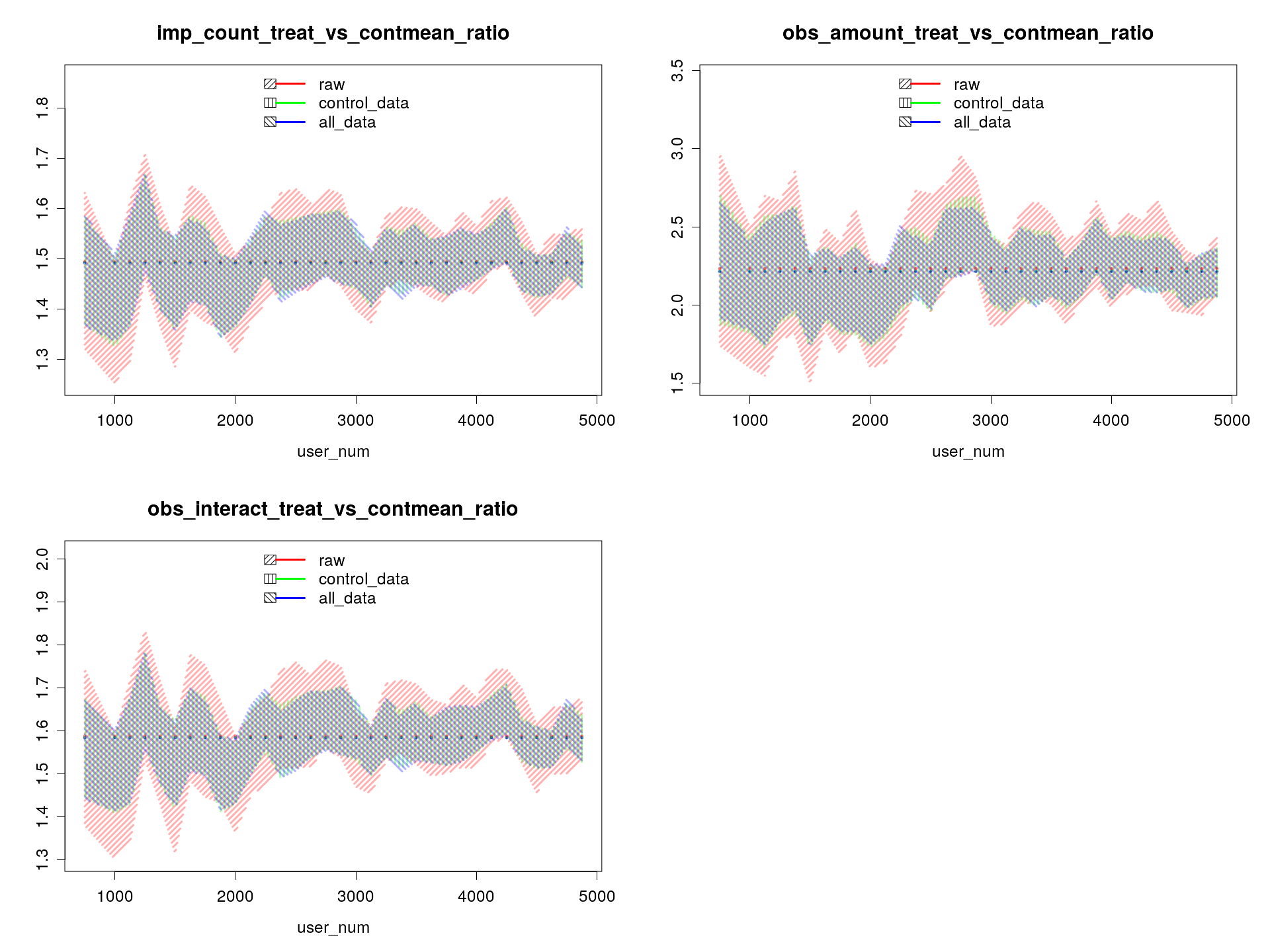}
	\caption{Comparing the CI length convergence across varying sample sizes.}
	\label{fig:mean_ratio_ci_convg_comparison_v15}
\end{figure}

\section{Discussion}
\label{section:discussion}

This paper developed a method to decrease the variance of a wide class of
experiment effects metrics while leaving the estimator (asymptotically) unbiased.

Our goal was to keep the method very flexible in two ways. (1) The method is applicable
to a wide variety of metrics. (2) The interpolation/prediction function used in the
method can be quite arbitrary i.e. any machine learning algorithm could be used for that purpose.
The second requirement is a useful one in the context of technology industry at the moment where
many powerful complex machine learning algorithms have been developed. However their
statistical properties are almost a black-box to the practitioners.
Given any such predictive model, we introduced adjusted estimators
which remain asymptotically unbiased.
The adjusted estimator also includes extra parameters which can be optimized for a
given predictive model.

\begin{comment}
We found close forms and approximated the reduction in the variance
using the optimal parameters in Theorem \ref{theo-var-g2} for general metrics involving one
response (including ratio metrics) and offered numerical equations to be optimized
for the metrics involving more than one response (e.g. ratio of ratios). Theorems \ref{theo-multi-ver1} and \ref{theo-multi-ver2}
extend this method to a much general class of estimators.
\end{comment}

We think these methods can have other applications beyond experiment analysis.
For example we can use these adjustments also
in the context of observational studies to estimate casual effects since
the adjustments perform balancing with respect to the prescribed
auxiliary variables.

%%%%%%%%%%%%%%%%%%%%%%%%%%%%%%%%%%%
%%%%%%%%%%%%%%%%%%%%%%%%%%%%%%%%%%%
%%%%%%%%%%%%%%%%%%%%%%%%%%%%%%%%%%%

\section*{Acknowledgement}
The first author (Reza Hosseini) would like to thank Jonathan Hennessy (Google)
for numerous fruitful discussions and for going over the first draft.

\appendix

\section{Proofs}
\label{appendix:proofs}

\begin{proof} (Proof of Theorem \ref{theo-var-g})
	\begin{itemize}
		\item[(a)]
			Define $f(H, Y; \theta) = g(Y) - \theta g(H)$. Also assume that the
			bivariate distribution of $(H, Y)$ has the mean $\mu=(\mu_H, \mu_Y)$.
			Then by applying the Taylor series approximation to f (see \cite{dasgupta-2008}):
			\[
				f(H, Y) =
					f(\mu) + \nabla (f)(\mu).[(H, Y) - \mu] =
					f(\mu) + \pd{f}{h}(\mu) (H - \mu_H) + \pd{f}{y}(\mu) (Y - \mu_Y).
			\]
			We can approximate the variance of $f(H, Y)$ by
			\[\var{f(H, Y; \theta)} \approx \pd{f}{h}(\mu)^2 \var{H} - 2 \pd{f}{h}(\mu) \pd{f}{y}(\mu) \cov(Y, H) + \pd{f}{y}(\mu)^2 \var{Y}.\]
			By replacing
			\[\pd{f}{h}(\mu) = -\theta g'(\mu_H), \; \pd{f}{y}(\mu) = g'(\mu_Y),\]
			we get:
			\[\var{f(H, Y; \theta)} = \theta^2 (g'(\mu_H) \sigma_H)^2 + (-2 g'(\mu_H) g'(\mu_Y)cov(Y, H))\theta + (g'(\mu_Y) \sigma_Y)^2.\]
			The above expression is a 2nd-order convex polynomial of the form $p(\theta) = a\theta^2 + b \theta + c$
			which is minimized at $-b / (2a)$ with minimum being $-b^2/4a + c$. Therefore
			\[\argmin_{\theta} f(X, H; \theta) = \frac{g'(\mu_Y)}{g'(\mu_H)} \frac{\cov(Y, H)}{\sigma_H^2}\]

      \item[(b)] The minimum can be obtained by replacing $\theta$ from (a).
      For the second part, note that by Taylor series approximation we have
      \[\var{g(Y)} = g'(\mu_Y) \var{Y},\] and the result follows.

      \end{itemize}

\end{proof}

\begin{proof} (Proof of Theorem \ref{theo-var-g2})

	To prove (a) and (b) define
	\[f(Y, H, \ystar, \hstar; \theta) = g(Y) - \theta g(H) - \big ( g(\ystar) - \theta g(\hstar) \big ) \]
	Then after by Taylor series expansion, we can approximate $f$
	\[f(Y, H, \ystar, \hstar; \theta) \approx f(\mu) + \nabla (f)(\mu).[(Y, H, \ystar, \hstar) - \mu].\]
	Therefore
	\begin{align*}
	\var{T} & \approx \sum_{U \in \{Y, H, \ystar, \hstar\}}  \pd{f}{u}(\mu)^2 \var{U} -
	2 \pd{f}{h}(\mu)\pd{f}{y}(\mu)\covyh - 2 \pd{f}{h^\star}(\mu)\pd{f}{y^\star}(\mu)\covystarhstar & \\
	 & = a \theta^2 + b\theta + c, &
	\end{align*}
	where,
	\begin{align*}
		a &= g'(\muh)^2 \sigma_H^2 + g'(\muhstar)^2 \sigmahstar^2 &  \\
		b &=  -2 \big ( g'(\muh)g'(\muy)\covyh + g'(\muhstar)g'(\muystar)\covystarhstar) \big ) & \\
		c &=  g'(\muy)^2 \sigmay^2 + g'(\muystar)^2 \sigmaystar^2.
	\end{align*}
	This is a quadratic function of $\theta$ for which the argmin is equal to
	$-b/(2a)$ and the minimum is equal to $-b^2/(4a) + c$ which gives (a) and (b).\\
	To prove (c), note that if we let
	\begin{align*}
		a_1 & = g'(\muh)g'(\muy) \cov(Y, H) & \\
		a_2 & = g'(\muhstar)g'(\muystar) \cov(\ystar, \hstar) & \\
		b_1 & = g'(\muh)^2 \sigmah^2 & \\
		b_2 & = g'(\muhstar)^2 \sigmahstar^2 &
	\end{align*}
	Then in (a) we showed $\argmin_{\theta} \var{T} \approx \frac{a_1 + a_2}{b_1 + b_2}$
	and from Theorem \ref{theo-var-g}
	we have $\theta_i = a_i / b_i,\;i=1,2$. The proof is complete by noting
	\[\min \{\frac{a_1}{b_1}, \frac{a_2}{b_2}\} \leq \frac{a_1 + a_2}{b_1 + b_2} \leq \max \{\frac{a_1}{b_1}, \frac{a_2}{b_2}\}\]

	To prove (d) note that by the assumptions given:
	 \[g'(\muh) \sigmah = g'(\muhstar) \sigmahstar \mbox{ and } \rho g'(\muy)\sigmay = \rho g'(\muystar)\sigmaystar,\]
	we conclude
	$a_1 = a_2$ and $b_1 = b_2$. This allows us to rewrite $\delta$ as
	$\delta = \frac{(a_1 + a_2)^2}{b_1 + b_2} = \frac{a_1^2}{b_1} + \frac{a_2^2}{b_2}.$
	By Part (b), we have
	\begin{align*}
		\var{T} & = g'(\muy)^2 \sigmay^2 + g'(\muystar)^2 \sigmaystar^2 - (a_1 + a_2)^2/(b_1 + b_2) & \\
		& =  g'(\muy)^2 \sigmay^2 + g'(\muystar)^2 \sigmaystar^2 - a_1^2/b_1 - a_2^2/b_2 & \\
		& = g'(\muy)^2 \sigmay^2  - \frac{a_1^2}{b_1} +  g'(\muystar)^2 \sigmaystar^2  - \frac{a_2^2}{b_2} & \\
		& = g'(\muy)^2 \sigmay^2 - \rho^2 g'(\muy)^2 \sigmay^2 + g'(\muystar)^2 \sigmaystar^2 - {\rhostar}^2 g'(\muystar)^2 \sigmaystar^2 & \\
		& = g'(\muy)^2 \sigmay^2 (1 - \rho^2) + g'(\muystar)^2 \sigmaystar^2 (1 - {\rhostar}^2) =  (1-\rho^2) \var{g(Y)} + (1-{\rhostar}^2) \var{g(\ystar)} & \\
		& \leq (1 - \min\{\rho^2, {\rhostar}^2\})  \var{g(Y)} +  (1 - \min\{\rho^2, {\rhostar}^2\}) \var{g(\ystar)} & \\
		& =  (1 - \min\{\rho^2, {\rhostar}^2\})  \big ( \var{g(Y)} + \var{g(\ystar)} \big ) = (1 - \min\{\rho^2, {\rhostar}^2\}) \var{T_0}.
	\end{align*}

\end{proof}

\begin{proof} (Proof of Corollary \ref{corol-var-reduction})

The adjusted estimator is equal to
\begin{align}
	\taua &= \hat{\tau}_g - {\color{\colorone}\theta \Big ( g(\fbart) - g(\fbarc) \Big )} \label{eqn-bias-form}\\
	& = g(\ybart) - g(\ybarc) - {\color{\colorone}\theta \Big ( g(\fbart) - g(\fbarc) \Big ) } \nonumber \\
	& = \Big( g(\ybart)  - {\color{\colorone}\theta (g(\fbart)} \Big) -
	\Big( g(\ybarc) -  {\color{\colorone} \theta g(\fbarc)} \Big) \label{eqn-var-form}
 \end{align}
The form in Equation \ref{eqn-bias-form} was convenient for showing unbiasedness.
And we use the form in Equation \ref{eqn-var-form} for deriving the variance.
Now we can use Theorem \ref{theo-var-g2} by letting
$Y = \bar{Y}_t,\;\ystar=\bar{Y}_c,\; H = \fbart,\;\hstar = \fbarc$, and
also noting that
\[\rho_{c} = \cor(Y_c, H_c) = \cor(\bar{Y}_c,  \bar{H}_c)\;\; \rho_t = \cor(\bar{Y}_t, \bar{H}_t)\]

\end{proof}

\section{Simulations for various scenarios}
\label{appendix:sims}

Here we perform a few more simulation studies for various scenarios.

\subsection{Simulation results for Sum Ratio}
\label{appendix:sum-ratio}

This subsection presents the results for the Sum Ratio Metric for the same
simulated data in Tables
  \ref{pred_accuracy_infoDf_contDataOnly_v15},
	\ref{pred_accuracy_infoDf_withTreatData_v15} and
	\ref{pred_accuracy_infoDf_allDataNoExptId_v15} (in the main text).

Figure \ref{fig:sum_ratio_v15} depicts the mean and the standard deviation of
the raw and adjusted estimators for Sum Ratio. Figure
\ref{fig:sum_ratio_sd_comparison_v15} depicts the ratio of the standard deviation
of the adjusted estimators versus the raw estimator.
We observe that the estimators remain unbiased while the variance is decreased,
albeit the decrease in the variance is less than the decrease for the Mean Ratio
case.

\begin{figure}[H]
	\centering
	\includegraphics[width=0.4\linewidth]{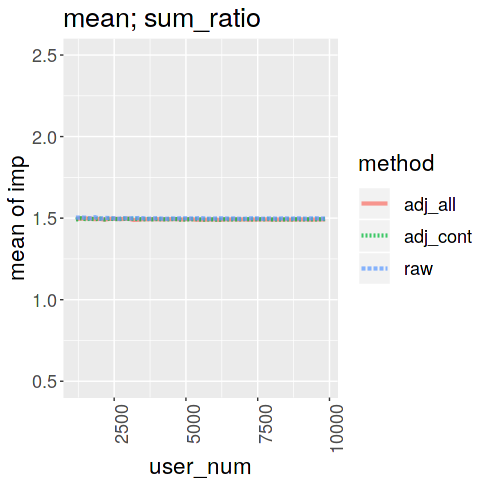}
	\includegraphics[width=0.4\linewidth]{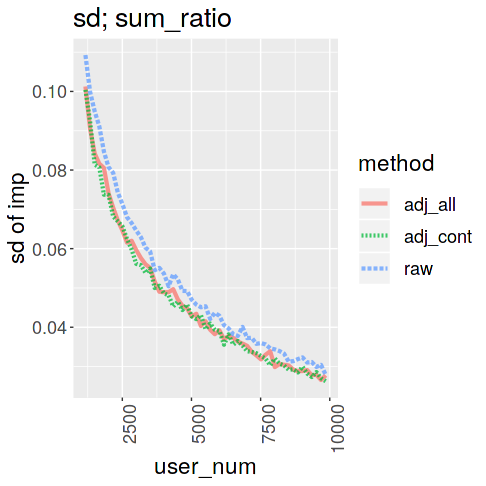}

	\includegraphics[width=0.4\linewidth]{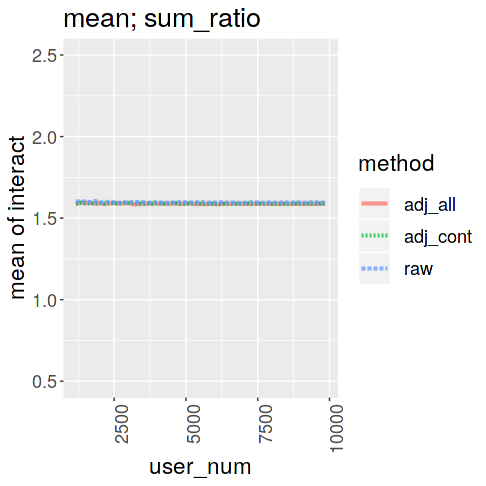}
	\includegraphics[width=0.4\linewidth]{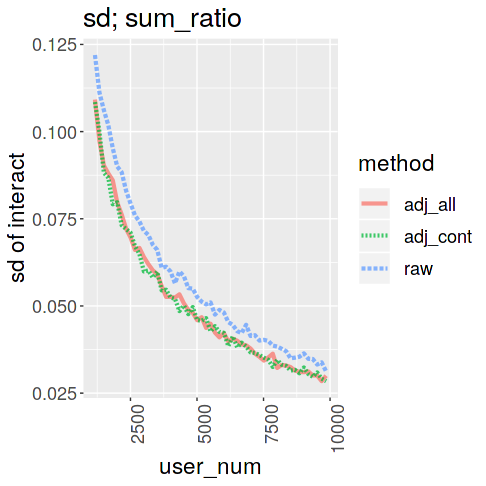}

	\includegraphics[width=0.4\linewidth]{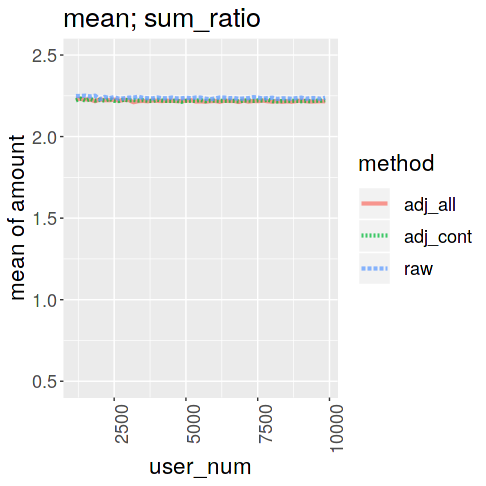}
	\includegraphics[width=0.4\linewidth]{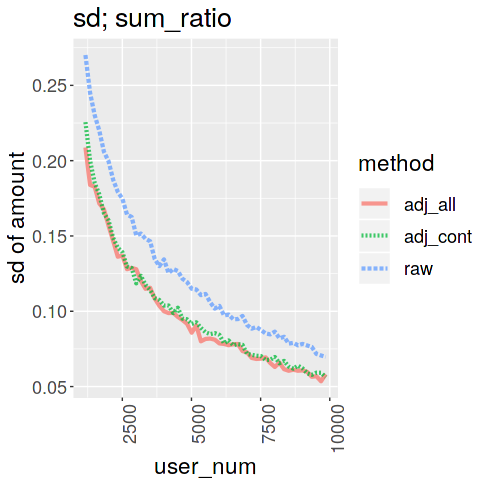}
	\caption{Sum Ratio Metric estimator's mean and variance across sample size.}
	\label{fig:sum_ratio_v15}
\end{figure}

\begin{figure}[H]
	\centering
	\includegraphics[width=0.8\linewidth]{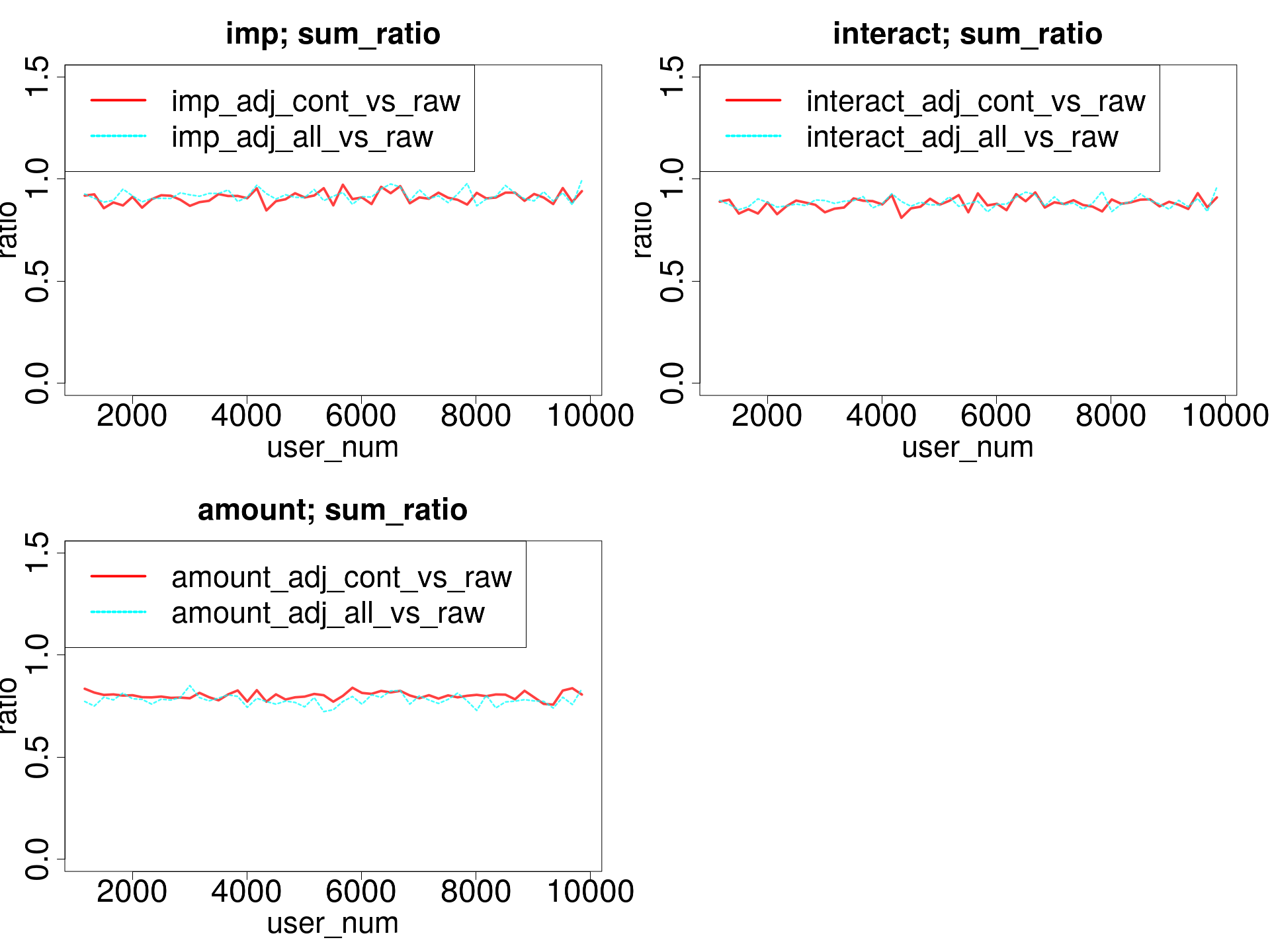}
	\caption{The ratio of standard deviations of adjusted and raw estimators for Sum Ratio.}
	\label{fig:sum_ratio_sd_comparison_v15}
\end{figure}

%\end{document}
\begin{comment}
\subsection{Simulation with smaller predictive power}
\label{appendix:smaller-predictive-power}

This subsection performs a simulation similar to the simulation in the main part
of the paper, but with smaller predictive power.
Tables \ref{pred_accuracy_infoDf_contDataOnly_v17} and
\ref{pred_accuracy_infoDf_withTreatData_v17}
show the predictive power of the models for this case and the expected reduction
in the variance.
Figure \ref{fig:mean_ratio_sd_comparison_v17} depict the ratio of the standard
deviation of the adjusted estimators versus the raw estimator. We observe that
the observed reduction in the estimator variance is again consistent to the
approximated reduction using the theory developed in the paper.

\input{pred_accuracy_infodf_contdataonly_v17.tex}
\input{pred_accuracy_infodf_withtreatdata_v17.tex}

\begin{figure}[H]
	\centering
	\includegraphics[width=1.05\linewidth]{mean_ratio_sd_comparison_v17.png}
	\caption{The ratio of standard deviations of adjusted and raw estimators for
	Mean Ratio. This is for the scenario with small predictive power.}
	\label{fig:mean_ratio_sd_comparison_v17}
\end{figure}
\end{comment}

\subsection{Simulation with no predictive power}
\label{appendix:no-predictive-power}

This subsection considers a simulation in which the auxiliary variables do not
have any prediction power.
It is desirable if the methods we introduced in this paper do not increase the
variance significantly, which might be the case due to the extra variance
introduced by the model uncertainty which generates some variance in the
adjustment. Tables \ref{pred_accuracy_infoDf_contDataOnly_v18} and
\ref{pred_accuracy_infoDf_withTreatData_v18} show the predictive power of the
models for this case and the expected reduction in the variance. Figure
\ref{fig:mean_ratio_v18} depicts the mean and the standard deviation of the raw
and adjusted estimators.
Figure \ref{fig:mean_ratio_sd_comparison_v18} depict the ratio of the standard
deviation of the adjusted estimators versus the raw estimator.
Fortunately we cannot observe any increase in the variance using the adjusted
methods.

% latex table generated in R 3.4.1 by xtable 1.8-2 package
% Mon Jan 14 11:16:40 2019
\begin{table}[ht]
\centering
\begin{tabular}{|r|l|r|r|r|}
  \hline
 & model\_formula & cv\_cor & theta & sd\_ratio \\
  \hline
1 & imp\_count: gender+country & 0.03 & 0.70 & 1.00 \\
  2 & obs\_interact: gender+country & -0.03 & 0.00 & 1.00 \\
  3 & obs\_amount: gender+country & -0.04 & 0.00 & 1.00 \\
   \hline
\end{tabular}
\caption{Prediction accuracy for various responses using only control data.} 
\label{pred_accuracy_infoDf_contDataOnly_v18}
\end{table}

% latex table generated in R 3.4.1 by xtable 1.8-2 package
% Mon Jan 14 11:16:40 2019
\begin{table}[ht]
\centering
\begin{tabular}{|r|l|r|r|r|}
  \hline
 & model\_formula & cv\_cor & theta & sd\_ratio \\
  \hline
1 & imp\_count: gender+country+expt\_id & 0.16 & 1.00 & 0.99 \\
  2 & obs\_interact: gender+country+expt\_id & 0.14 & 0.67 & 0.99 \\
  3 & obs\_amount: gender+country+expt\_id & 0.19 & 0.91 & 0.98 \\
   \hline
\end{tabular}
\caption{Prediction accuracy for various responses when using all the data in building the model and the experiment label. However we assume all data is on control arm when predicting.}
\label{pred_accuracy_infoDf_withTreatData_v18}
\end{table}

\begin{figure}[H]
	\centering
	\includegraphics[width=0.4\linewidth]{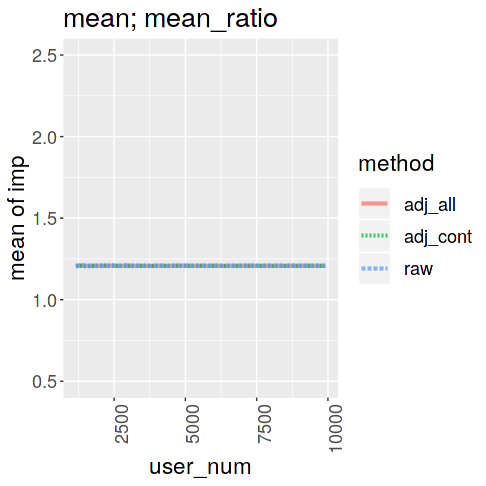}
	\includegraphics[width=0.4\linewidth]{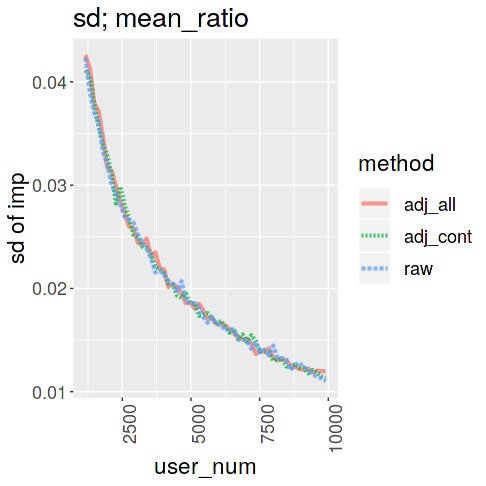}

	\includegraphics[width=0.4\linewidth]{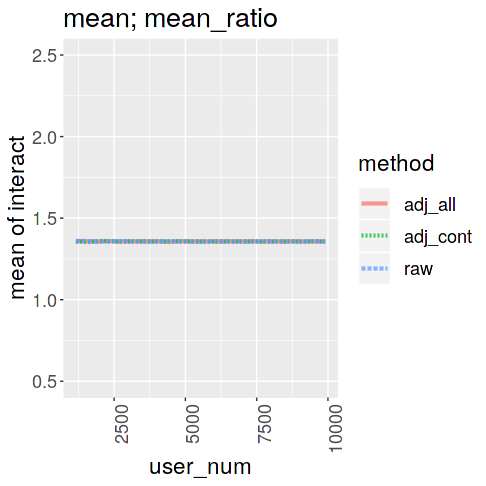}
	\includegraphics[width=0.4\linewidth]{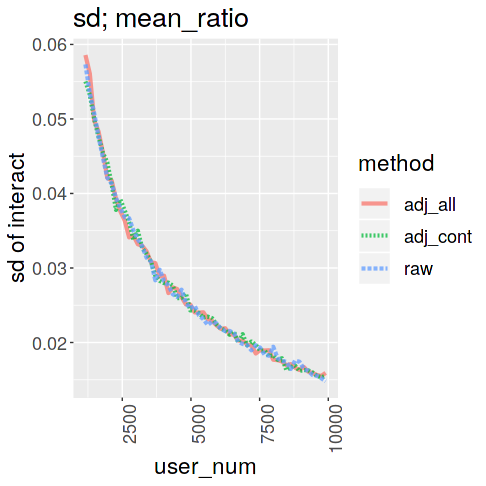}

	\includegraphics[width=0.4\linewidth]{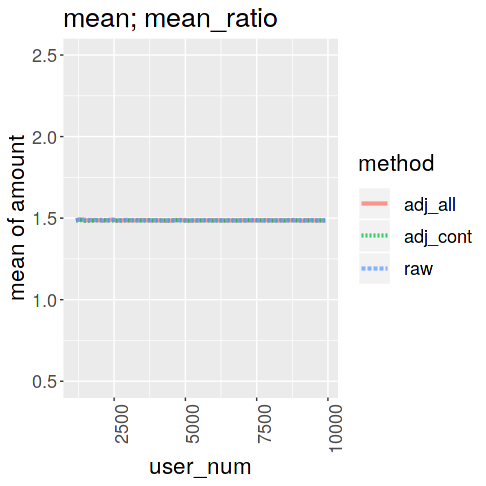}
	\includegraphics[width=0.4\linewidth]{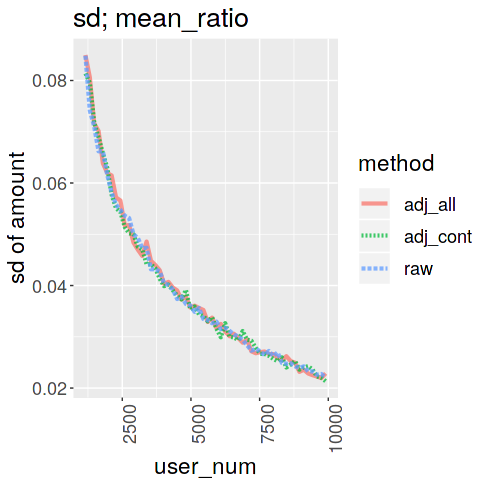}
	\caption{Mean Ratio Metric estimator's mean and variance across sample size.
	This is for the scenario with no predictive power.}
	\label{fig:mean_ratio_v18}
\end{figure}

\begin{comment}
\begin{figure}[H]
	\centering
	\includegraphics[width=1.05\linewidth]{mean_ratio_v18.png}
	\caption{Mean Ratio Metric estimator's mean and variance across sample size.
	This is for the scenario with no predictive power.}
	\label{fig:mean_ratio_v18}
\end{figure}
\end{comment}

\begin{figure}[H]
	\centering
	\includegraphics[width=0.8\linewidth]{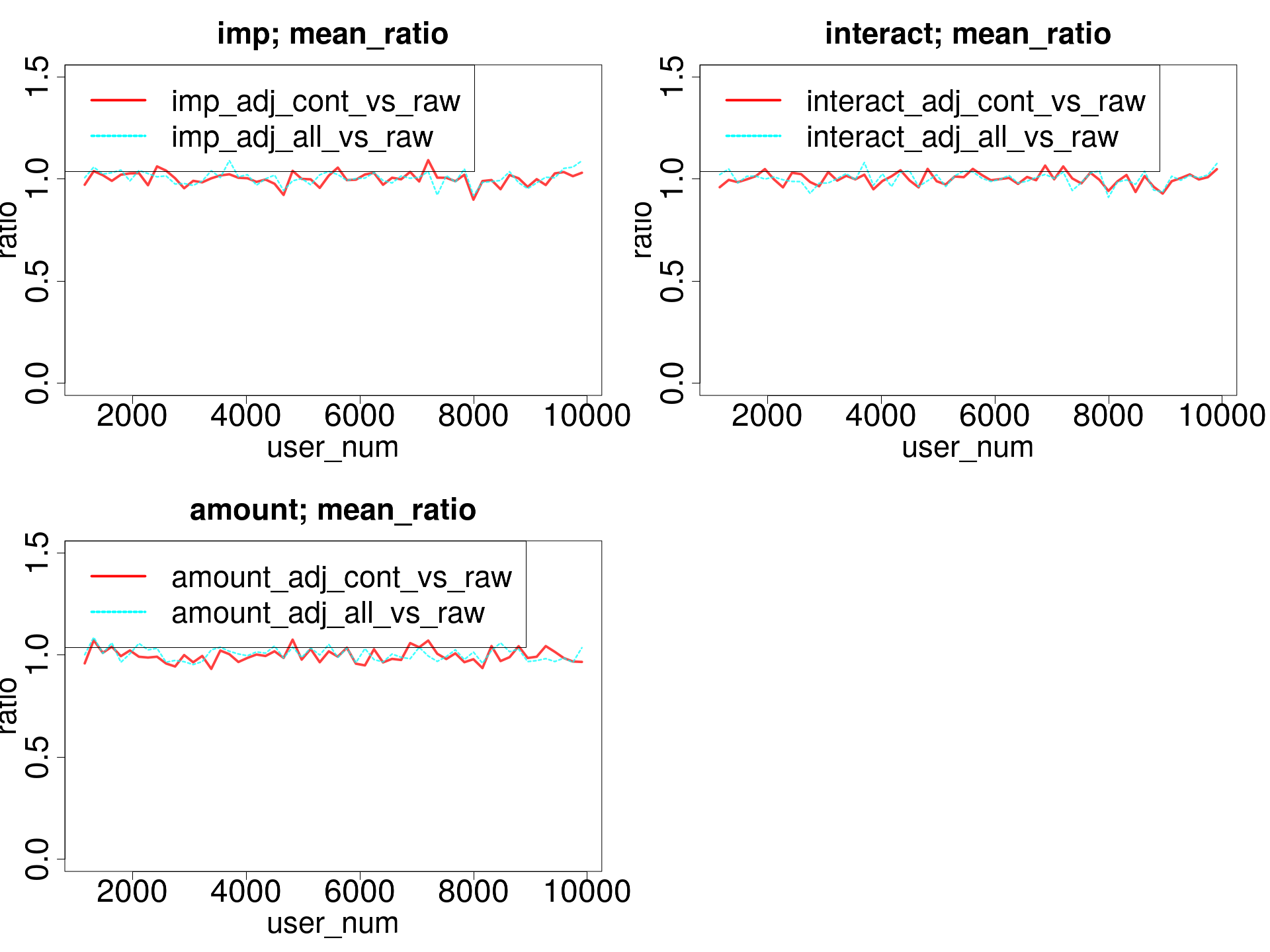}
	\caption{The ratio of standard deviations of adjusted and raw estimators for
		Mean Ratio. This is for the scenario with no predictive power.}
	\label{fig:mean_ratio_sd_comparison_v18}
\end{figure}

\begin{figure}[H]
	\centering
	\includegraphics[width=0.4\linewidth]{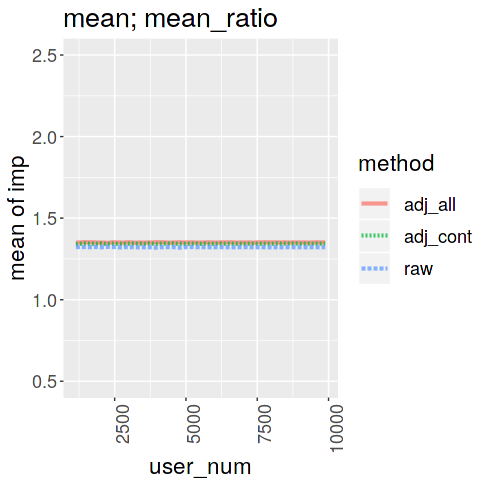}
	\includegraphics[width=0.4\linewidth]{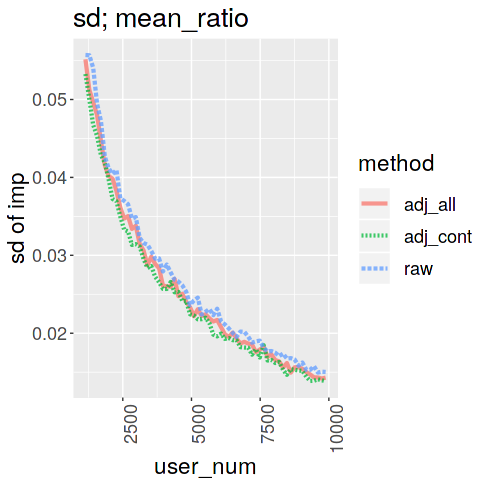}

	\includegraphics[width=0.4\linewidth]{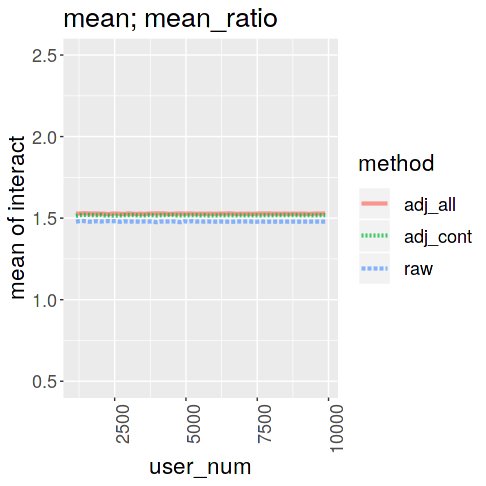}
	\includegraphics[width=0.4\linewidth]{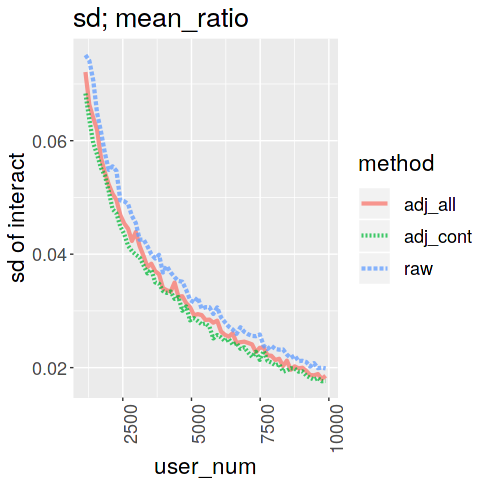}

	\includegraphics[width=0.4\linewidth]{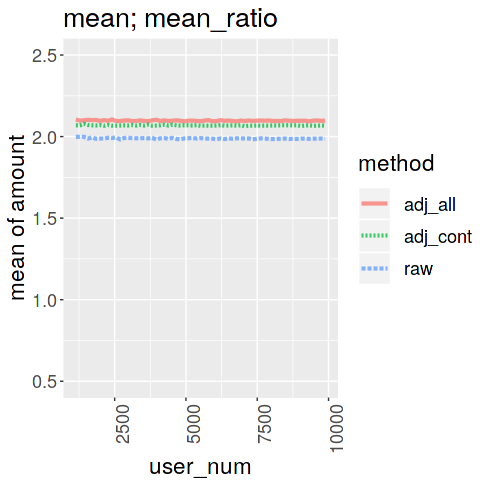}
	\includegraphics[width=0.4\linewidth]{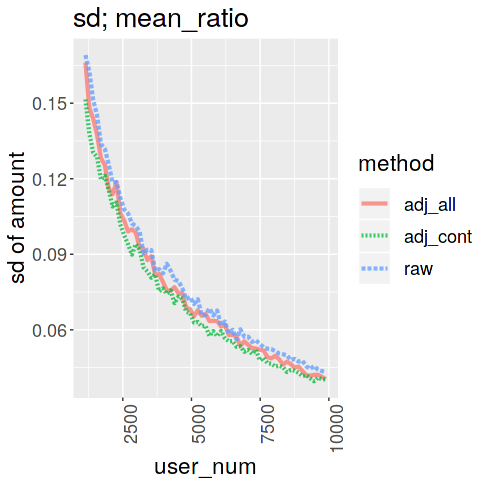}
	\caption{Mean Ratio Metric estimator's mean and variance across sample size.
	This is for the scenario with experiment impacting the user populations
	appearing in the data.}
	\label{fig:mean_ratio_v19}
\end{figure}

\begin{figure}[H]
	\centering
	\includegraphics[width=0.999\linewidth]{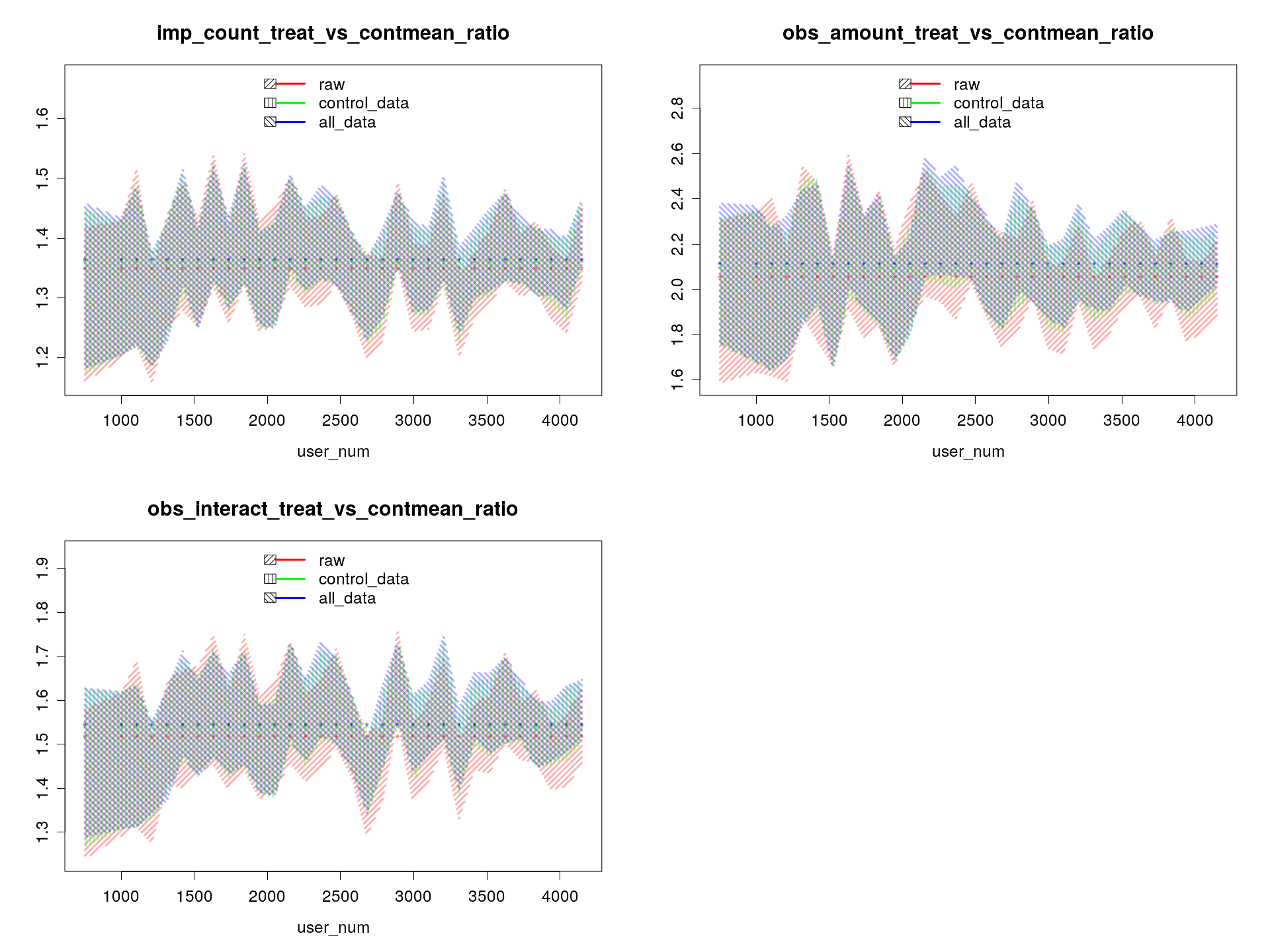}
	\caption{Comparing the CI length convergence across varying sample sizes.}
	\label{fig:mean_ratio_ci_convg_comparison_v19}
\end{figure}

\subsection{Simulation results for Ratio of Mean Ratios}
\label{appendix:ratio-of-ratio}
We use the same simulation settings as the simulation in the main text here.
The result is given in Figure \ref{fig:ratio_of_mean_ratios_ci_convg_comparison_v15}
which show the adjusted estimator is unbiased while decreasing the variance.
\begin{figure}[H]
	\centering
	\includegraphics[width=0.6\linewidth]{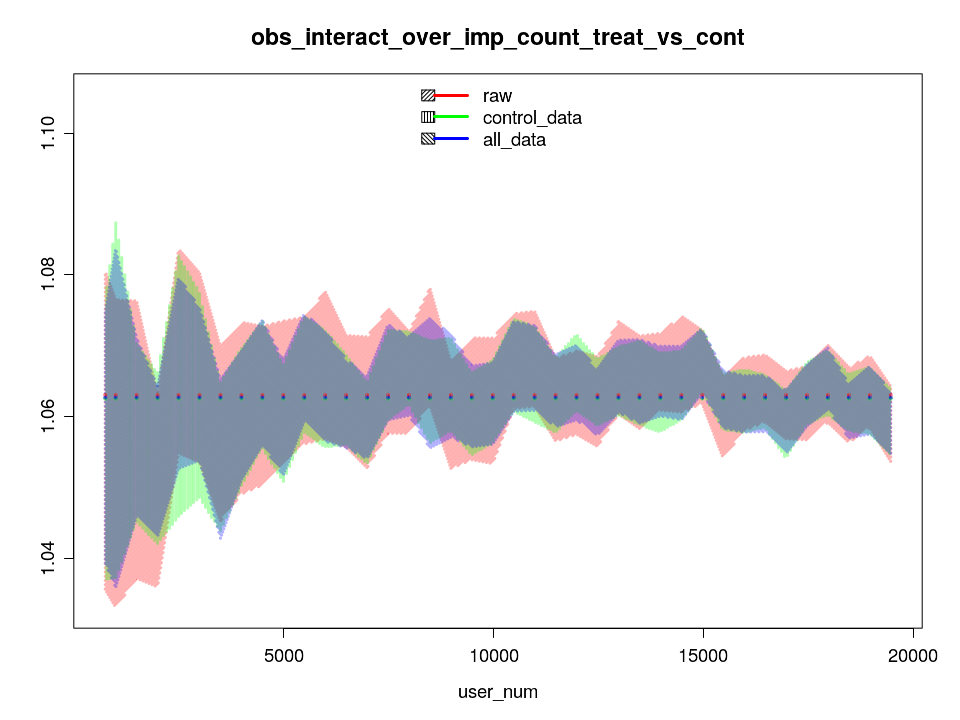}
	\caption{The ratio of mean ratios.}
	\label{fig:ratio_of_mean_ratios_ci_convg_comparison_v15}
\end{figure}

\end{document}